\documentclass[final,onefignum,onetabnum]{siamart171218}
\usepackage{latexsym, amssymb, enumerate, amsmath}
\usepackage{defs}
\usepackage{fonts}
\usepackage{dsfont}
\usepackage{epsfig}
\usepackage{epstopdf}
\usepackage{booktabs}
\usepackage[normalem]{ulem}
\newtheorem{remark}{Remark}

\usepackage{cite}
\usepackage{multirow} 
\usepackage{subfig} 
\usepackage{enumitem}
\usepackage{algorithm}
\usepackage{algpseudocode}

\usepackage{xargs}                      % Use more than one optional parameter in a new commands

\usepackage[colorinlistoftodos,prependcaption]{todonotes}
\newcommandx{\mytodo}[2][1=]{\todo[inline,linecolor=blue,backgroundcolor=blue!25,bordercolor=blue,#1]{#2}}
\newcommandx{\tothink}[2][1=]{\todo[inline,linecolor=green,backgroundcolor=green!15,bordercolor=green!20!black,#1]{#2}}
\newcommandx{\mycomment}[2][1=]{\todo[inline,linecolor=red,backgroundcolor=red!25,bordercolor=red,#1]{#2}}

\graphicspath{{./figures/linesource_FOTime/},{./figures/nonuniform_FOTime/}}

\title{A Positive Asymptotic Preserving Scheme for Linear Kinetic Transport Equations%
	\thanks
	{This manuscript has been authored, in part, by UT-Battelle, LLC, under Contract No. DE-AC0500OR22725 with the U.S. Department of Energy. The United States Government retains and the publisher, by accepting the article for publication, acknowledges that the United States Government retains a non-exclusive, paid-up, irrevocable, world-wide license to publish or reproduce the published form of this manuscript, or allow others to do so, for the United States Government purposes. The Department of Energy will provide public access to these results of federally sponsored research in accordance with the DOE Public Access Plan (\texttt{http://energy.gov/downloads/doe-public-access-plan}).}
}

\date{\today}
\author{M. Paul Laiu%
\thanks{Computational Mathematics Group,
				Computer Science and Mathematics Division,
				Oak Ridge National Laboratory,
				Oak Ridge, TN 37831 USA, (\texttt{laiump@ornl.gov}).
Supported by the U.S. Department of Energy, under the SCGSR program administered by the Oak Ridge Institute for Science and Education under Contract No. DE-AC05-06OR23100.}       
\and
Martin Frank%
\thanks{
Karlsruhe Institute of Technology, 
           Steinbuch Center for Computing, 
           Hermann-von-Helmholtz-Platz 1,
           76344 Eggenstein-Leopoldshafen, Germany, (\texttt{martin.frank@kit.edu}).
}
\and
        Cory D. Hauck%
\thanks{Computational Mathematics Group,
				Computer Science and Mathematics Division,
				Oak Ridge National Laboratory,
				Oak Ridge, TN 37831 USA, (\texttt{hauckc@ornl.gov}).
This author's research was sponsored by the Office of Advanced Scientific
Computing Research and performed at the Oak Ridge National Laboratory,
which is managed by UT-Battelle, LLC under Contract No. DE-AC05-00OR22725.}
}

\begin{document}

\maketitle
\begin{abstract}
We present a positive and asymptotic preserving numerical scheme for solving linear kinetic, transport equations that relax to a diffusive equation in the limit of infinite scattering.
The proposed scheme is developed using a standard spectral angular discretization and a classical micro-macro decomposition.
The three main ingredients are a semi-implicit temporal discretization, a dedicated finite difference spatial discretization, and {\limname} limiters in the angular discretization.
Under mild assumptions on the initial condition and time step, the scheme becomes a consistent numerical discretization for the limiting diffusion equation when the scattering cross-section tends to infinity.  
The scheme also preserves positivity of the particle concentration on the space-time mesh and therefore fixes a common defect of spectral angular discretizations.
The scheme is tested on the well-known line source benchmark problem with the usual uniform material medium as well as a medium composed from different materials that are arranged in a checkerboard pattern. 
We also report the observed order of space-time accuracy of the proposed scheme.
\end{abstract}

\begin{keywords}
kinetic transport equations, diffusion limit, positive-preserving schemes, asymptotic preserving schemes, finite difference methods
\end{keywords}

\begin{AMS}
35B09, 35L40, 41A60, 65M06, 65M70, 82C70, 82D75 
\end{AMS}

%\linenumbers

%\listoftodos

%\tableofcontents

\section{Introduction}
\label{sec:intro}

Kinetic transport equations are widely used to model particle systems in many applications, including thermal radiative transfer \cite{Pomraning-1973,Mihalis-Mihalis-1999}, rarefied gas dynamics \cite{cercignani1988boltzmann}, plasmas \cite{hazeltine2018framework} and neutron transport \cite{Lewis-Miller-1984,Davison-1957}. 
These equations track the temporal evolution of a particle distribution function in a position-velocity phase space. 

For kinetic equations that model propagation through a background medium, the kinetic distribution function is often approximated by the solution of a much simpler diffusion equation.  Such an approximation is accurate when the dynamics of the particle system are dominated by scattering interactions with the medium.  However, many problems of interest are ``multiscale" in the sense that scattering and other material cross-sections may vary in space by several orders of magnitude. 
In regions of moderate scattering, the diffusion approximation may not be sufficiently accurate, while in strongly scattering regions, classical numerical schemes for kinetic equations must resolve collisional length scales, making them computationally prohibitive.  In addition, explicit time integrators for kinetic equations in scattering dominated regimes require very small time steps in oder to maintain accuracy and stability, while implicit time integrators can be exceptionally stiff.
For these reasons, it is desirable to solve multi-scale problems with numerical schemes that are consistent in both the kinetic and diffusive regions and uniformly stable under a reasonable time-step restriction. Such schemes are often referred to as ``asymptotic preserving" (AP) schemes \cite{Jin-1999}.

Asymptotic preserving schemes for kinetic equations with diffusion limits were first considered in the context of steady-state neutron transport \cite{Larsen-Morel-Miller-1987,Larsen-Morel-1989}.  Since then, a variety of approaches have been taken, including discontinuous Galerkin methods \cite{Larsen-Morel-1989,Adams-2001,Guermond-Kanschat-2010,McClarren-Lowrie-2008,hauck2010methods}, methods based on even-odd parity
\cite{Miller-Warren-1991,Jin-Pareschi-Toscani-1998,Jin-Pareschi-Toscani-2000,Seibold-Frank-2014}, micro-macro decompositions \cite{Lemou-Mieussens-2008,Liu-Mieussens-2010}, numerical fluxes that depend on the scattering cross-section \cite{Jin-Levermore-1996,Buet-Despres-Franck-2012}, temporal regularization \cite{Jin-1999,Hauck-Lowrie-2009}, well-balanced schemes \cite{Gosse-Toscani-2002,Gosse-Toscani-2004}, and unified gas-kinetic schemes \cite{Xu-Huang-2010,Mieussens-2013}.  Many of these approaches are related or overlapping, and despite any differences, they all seek to address two fundamental issues:  first, that the numerical dissipation induced by the discretization of the hyperbolic advection operator in the kinetic equation must be controlled in the diffusion limit;  second, that for time-dependent problems, stiffness must be overcome either with semi-implicit time integrators \cite{boscarino2013implicit} or with fully implicit time integrators that leverage acceleration and/or preconditioning techniques \cite{adams2002fast,larsen2010advances}.

Deterministic numerical simulations of kinetic equations require discretization in space, velocity, and time. 
Spherical harmonic ({\pn}) methods \cite{Case-Zweifel-1967,Lewis-Miller-1984,Pomraning-1973} discretize the angular component of the velocity using a polynomial approximation with coefficients that are functions of space and time.
The kinetic equations are then converted using the standard Galerkin approach into a system of reduced equations for these coefficients.
As with other spectral methods, approximate solutions generated in this way converge spectrally to the solution of the kinetic equation when the latter is sufficiently smooth.
However, when the solution to the kinetic equation is discontinuous, the {\pn} method produces oscillatory solutions which may cause the approximation of the particle concentration (defined as the integral of the kinetic distribution over the velocity variable) to become negative.

In \cite{Hauck-McClarren-2010}, filtering techniques \cite{Gottlieb-Gottlieb-Hesthaven-2007,Guo-1998} for mitigating the Gibbs phenomena in spectral approximations were proposed as a way to reduce oscillations in the solution of the {\pn} equations.
Later in \cite{Radice-2013}, this idea was used to derive a system of modified \pn~equations, referred to as Filtered {\pn} equations.
While filtering significantly reduces oscillations in the profile of the particle concentration, negative values are still possible.
Thus in \cite{LH-2016}, several positive-preserving schemes were proposed to augment the filtering strategy.
These schemes combine a second-order, explicit, finite-volume discretization in space and time \cite{AHT10,GH12} with limiters that force the spectral approximation in angle to be non-negative on a finite set of quadrature points in the angular domain.  The schemes do preserve positivity of the particular concentration.  However, the limiters may reduce accuracy or be computationally expensive, and the finite-volume discretization is not AP.

In this paper, we propose a positive-preserving AP scheme for solving the {\fpn} equations.
The proposed scheme follows the approach proposed in \cite{Lemou-Mieussens-2008} and analyzed in \cite{Liu-Mieussens-2010}, where a one-dimensional kinetic equation was solved using a classical micro-macro decomposition \cite{liu2004boltzmann} of the kinetic distribution.
Here we use the micro-macro decomposition to formulate the {\fpn} equations as a coupled system for the expansion coefficients that correspond to the macro and micro parts of the kinetic distribution.
The numerical scheme we use to solve the micro-macro system involves three main ingredients -- a semi-implicit temporal discretization, a dedicated finite difference spatial discretization, and {\limname} limiters in the angular discretization.  The space-time discretization is designed so that the {\limname} limiters in the angular discretization are physically reasonable, but also less strict than the pointwise limiters used in \cite{LH-2016}.
In designing the spatial discretization, we focus on a simplified geometry that allows for a formulation in two space dimensions.  However, we also discuss how to extend the proposed numerical scheme to the full three-dimensional setting.

The remainder of the paper is organized as follows.  In Section~\ref{sec:Kinetic_FPN_diffusion_limit}, we introduce the linear kinetic equation, the {\fpn} equations, their diffusion limits, and the derivation of the associated micro-macro systems.
In Section~\ref{sec:Full_disc}, we present the space-time discretization for the micro-macro {\fpn} system and show that, under mild assumptions on the initial condition and time step, the fully discretized scheme gives a consistent explicit numerical scheme for the diffusion limit when the scattering cross-section tends to infinity.
In Section~\ref{sec:Positivity}, we give sufficient conditions for preserving positivity of the particle concentration and we detail the approach, including the {\limname} limiters and time-step restriction needed to enforce these conditions. 
We test the proposed scheme on two benchmark problems in both kinetic and diffusive regimes and report the results in Section~\ref{sec:num_results}.
Conclusions and discussion are given in Section~\ref{sec:conclusion}.

\section{Linear kinetic equations, {\fpn} equations, and the diffusion limits}
\label{sec:Kinetic_FPN_diffusion_limit}

\subsection{Linear kinetic equation and its diffusion limit}
%\label{subsec:transport_PN}

We consider the linear kinetic transport equation 
\begin{equation}
\p_t f + \Omega \cdot\nabla_r f = \sig{s}\fM - \sig{t}f\:,
%\p_t f + \Omega \cdot\nabla_r f = \frac{\sig{s}}{4\pi}\vint{f} - \sig{t}f\:,
%\p_t f + \Omega \cdot\nabla_r f = \frac{\sig{s}}{4\pi}\vint{f} - \sig{t}f - \cQ_{\textup{F}}(f)\:,
\label{eq:transport}
\end{equation}
for the kinetic distribution function $f = f(r,\Omega,t)$.  Here $r=(x,y,z)\in\bbR^3$ is the position; $\Omega=(\Omega_x,\Omega_y,\Omega_z)\in\bbS^2$ is the angle; $\sig{s}(r)\geq\sig{s}^{\min}>0$, $\sig{a}(r)\geq0$, and $\sig{t}=\sig{s}+\sig{a}$ are respectively the scattering, absorption, and total cross-sections; $\fM(r,t)=({4\pi})^{-1}\vint{f}$, where $\vint{\cdot}$ denotes integration over $\bbS^2$ with respect to $\Omega$, is the angular average of $f$.  We denote the particle concentration associated to $f$ by $\rho=\vint{f}=4\pi\fM$.
With appropriate initial and boundary conditions, \eqref{eq:transport} is known to have a unique solution \cite{Dautray-Lions-2000}.

Given $\epsilon >0$, letting $\sig{s} \to \epsilon ^{-1} \sig{s}$, $\sig{a} \to \epsilon \sig{a}$, and $t \to \epsilon^{-1} t$ in \eqref{eq:transport} 
leads to the scaled equation%
%\footnote{In \eqref{eq:transport_scaled}, the filtering term $\cQ_{\textup{F}}(f)$ is scaled such that it vanishes at the diffusion limit.}
\begin{equation}
\epsilon\p_t f + \Omega \cdot\nabla_r f = \frac{\sig{s}}{\epsilon}\left(\fM- f\right) - \epsilon\sig{a} f\:.
%\epsilon\p_t f + \Omega \cdot\nabla_r f = \frac{\sig{s}}{\epsilon}\left(\frac{1}{4\pi}\vint{f}- f\right) - \epsilon\sig{a} f\:.
%\epsilon\p_t f + \Omega \cdot\nabla_r f = \frac{\sig{s}}{\epsilon}\left(\frac{1}{4\pi}\vint{f}- f\right) - \epsilon\sig{a} f - \epsilon\cQ_{\textup{F}}(f)\:.
\label{eq:transport_scaled}
\end{equation}
It is well-known \cite{Habetler-Matkowsky-1975,Larsen-Keller-1974} that when $\epsilon\ll 1$, the kinetic distribution $f$ in \eqref{eq:transport_scaled} is given by $f=\fM+O(\epsilon)$.
Meanwhile, the particle concentration $\rho=4\pi \fM$ is governed approximately by a diffusion equation
\begin{equation}
\p_t \rho - \nabla_r \cdot\left(D \nabla_r \rho\right) + \sig{a} \rho = O(\epsilon) \:,
\label{eq:diffusion}
\end{equation}
where the matrix of diffusion coefficients $D$ is given by 
\begin{equation}
D = \frac{1}{4\pi\sig{s}}\diag
\left({\vint{\Omega_x^2}},\, {\vint{\Omega_y^2}},\, {\vint{\Omega_z^2}}\right) = \frac{1}{3\sig{s}} I_{3\times 3}\:.
\label{eq:D_def}
\end{equation}
When $\epsilon\to0$, the right-hand side of \eqref{eq:diffusion} vanishes, and the resulting equation \eqref{eq:diffusion} is referred to as the diffusion limit for \eqref{eq:transport_scaled}.

\subsection{{\pn} and {\fpn} equations and their diffusion limit}
The {\pn} method \cite{Case-Zweifel-1967,Lewis-Miller-1984} approximates the kinetic distribution by a polynomial expansion in $\Omega$ with coefficients that are functions of space and time.
When the angular space is the unit sphere, spherical harmonics are commonly used as the basis for spectral approximations.
Specifically, let $\bbP_N(\bbS^2)\subset L^2(\bbS^2)$ be the vector space of polynomials in $\Omega$ with degree at most $N$, and let $\bfm\colon \bbS^2 \to \bbR^n$, where $n=\text{dim}(\bbP_{N}(\bbS^2))$, be a vector-valued function that takes the form $\bfm:=[m_0^0,m_1^{-1},m_1^{0},m_1^{1},\dots]^T$, where $m_\ell^k$ denotes the real-valued spherical harmonic of degree $\ell$ and order $k$, normalized such that $\vint{m_{\ell}^k m_{\ell^\prime}^{k^\prime}}=\delta_{\ell,\ell^\prime}\cdot\delta_{k,k^\prime}$ with $\delta_{\ell,\ell^\prime}$ the Kronecker delta function.
For example, the first few components of $\bfm$ are 
\begin{equation}
m_0^0 = \sqrt{\frac{1}{4\pi}}\:,\quad m_1^{-1} = \sqrt{\frac{3}{4\pi}} \Omega_y \:,\quad m_1^{0} = \sqrt{\frac{3}{4\pi}} \Omega_z\:,\quad m_1^{1} = \sqrt{\frac{3}{4\pi}} \Omega_x\:.
\label{eq:sph_har}
\end{equation}
The components of $\bfm$ form an orthonormal basis of $\bbP_{N}(\bbS^2)$, and the scaled {\pn} equations corresponding to \eqref{eq:transport_scaled} are given by
\begin{equation}
\epsilon\p_t\upn+ \Vint{\bfm\bfm^T \Omega} \cdot \nabla_r \upn = -\frac{\sig{s}}{\epsilon} R \upn - \epsilon\sig{a} \upn  \:,
  \label{eq:pn}
\end{equation}
where $R=\diag([0,1,\dots,1])$, and
\begin{equation}
\Vint{\bfm\bfm^T\Omega} \cdot \nabla_r := \Vint{\bfm\bfm^T \Omega_x} \p_x + \Vint{\bfm\bfm^T \Omega_y} \p_y + \Vint{\bfm\bfm^T \Omega_z} \p_z\:.
\end{equation}
The solution $\upn\colon \bbR^3 \times \bbR^+ \to \bbR^n$ to \eqref{eq:pn} is an approximation to the spectral expansion coefficients of $f$, and the initial condition is given by $\upn(r,0):=\Vint{\bfm f(r,\Omega,0)}$.
The {\pn} equations form a symmetric, linear hyperbolic system of PDEs.

When the solution to \eqref{eq:transport_scaled} is not smooth, the {\pn} method produces oscillatory solutions \cite{Brunner-2002, GH12}.
To reduce oscillations, a filtering term was introduced into \eqref{eq:pn} in \cite{Radice-2013}, resulting in the following system of modified equations:
\begin{equation}
\epsilon\p_t\ufpn+\Vint{\bfm\bfm^T \Omega} \cdot \nabla_r \ufpn = -\frac{\sig{s}}{\epsilon} R \ufpn - \epsilon\sig{a} \ufpn -\epsilon{\sigF}F\ufpn \:,
  \label{eq:fpn}
\end{equation}
where ${\sigF}>0$ is a filtering parameter, the filtering matrix $F\in\bbR^{n\times n}$ is a diagonal matrix with elements $F_{(\ell,k),(\ell,k)} = - \ln\left(\kappa\left(\frac{\ell}{N+1}\right)\right)$, and $\kappa\colon\bbR^+\rightarrow[0,1]$ is a filter function with $\kappa(0)=1$. (See, for example, \cite{FHK-14} for a detailed definition of $\kappa$.)
The solution $\ufpn \colon \bbR^3 \times \bbR^+ \to \bbR^n$ to \eqref{eq:fpn} is also an approximation to the spectral expansion coefficients of $f$, and the initial condition is given by $\ufpn(r,0)=\Vint{\bfm f(r,\Omega,0)}$.
The modified equations \eqref{eq:fpn}, referred to as the filtered {\pn} ({\fpn}) equations \cite{Radice-2013,FHK-14}, also form a linear hyperbolic system.
Analogous to \eqref{eq:diffusion}, the diffusion limit of \eqref{eq:fpn} is given by
\begin{equation}
\p_t \bufpn - \nabla_r \cdot\left(D \nabla_r \bufpn \right) + \sig{a} \bufpn  = 0 \:,
\label{eq:diffusion_moment}
\end{equation}
where $D$ is as defined in \eqref{eq:D_def} and $\bufpn$ denotes the first component of $\ufpn$.
Note that the filtering term ${\sigF}F\ufpn$ in \eqref{eq:fpn} is scaled such that it vanishes as $\epsilon\to0$, since the solution is generally not oscillatory in the diffusion limit.

\subsection{The micro-macro decomposition}
\label{subsec:Micro_Macro}

For the scaled kinetic equation \eqref{eq:transport_scaled}, the kinetic distribution $f$ can be decomposed into $f(r,\Omega,t)=\fM(r,t) + \epsilon \fm(r,\Omega,t)$ (see, e.g., \cite{Lemou-Mieussens-2008} for details), where the macro component $\fM$ is constant with respect to $\Omega$ and the micro component $\fm$ satisfies $\vint{\fm}=0$.
The governing equations for $\fM$ and $\fm$ are 
\begin{subequations}
\begin{align}
\p_t \fM + \frac{1}{4\pi}\vint{\Omega \cdot\nabla_r \fm} + \sig{a} \fM &= 0\:,
\label{eq:macro}\\
\p_t \fm + \frac{1}{\epsilon}\Omega \cdot\nabla_r \fm - \frac{1}{4\pi\epsilon}\vint{\Omega \cdot\nabla_r \fm} + \sig{a} \fm 
%= - \frac{\sig{s}}{\epsilon^2} \fm - \vint{\cQ_{\textup{F}}(\fm)} - \frac{1}{\epsilon^2}\Omega \cdot\nabla_r \rho\:.
&= - \frac{\sig{s}}{\epsilon^2} \fm  - \frac{1}{\epsilon^2}\Omega \cdot\nabla_r \fM\:.
\label{eq:micro}
\end{align}

\end{subequations}

We apply a micro-macro decomposition to the {\fpn} equations \eqref{eq:fpn}.
Specifically, we decompose $\ufpn\in\bbR^n$ into the macro expansion coefficients $\uM\in \bbR$ and micro expansion coefficients $\um\in \bbR^{\nm}$, where $\nm:=n-1$ and $\ufpn=[\uM,\epsilon\um^T]^T$.
To simplify the notation, we drop subscripts and set  $\uF=\ufpn$ for the remainder of this paper.
We also let $\mM$ denote $m_0^0=(4\pi)^{-\frac{1}{2}}$ and let $\mm\colon\bbS^2\to\bbR^{\nm}$ denote the remaining components of $\bfm$, i.e., $\bfm=:[\mM,\mm^T]^T$.
Then, similar to \eqref{eq:macro}--\eqref{eq:micro}, $\uM$ and $\um$ are governed by the \textit{micro-macro system}
\begin{subequations}
\begin{align}
\label{eq:macro_moment}
\p_t \uM + \vint{\mM \mm^T \Omega} \cdot\nabla_r \um + \sig{a} \uM &= 0\:,\\
\label{eq:micro_moment}
\p_t \um + \frac{1}{\epsilon}\vint{\mm\mm^T \Omega} \cdot\nabla_r \um  + \sig{a} \um 
&= - \frac{\sig{s}}{\epsilon^2} \um  - {\sigF}\Fm\um - \frac{1}{\epsilon^2}\vint{\mm\mM\Omega} \cdot\nabla_r \uM\:,
\end{align}
where $\Fm\in\bbR^{\nm\times\nm}$ is formed by removing the first column and first row of $F$.
\end{subequations}

\section{Space-time discretization}
\label{sec:Full_disc}
We present a semi-implicit time discretization for the micro-macro system \eqref{eq:macro_moment}--\eqref{eq:micro_moment} in Section~\ref{subsec:Time_discr}.
In Section~\ref{subsec:reduced_kinetic_eqn}, we introduce a reduced two-dimensional linear kinetic equation that is considered in the remainder of this paper. 
The finite difference spatial discretization for the reduced equation is given in Section~\ref{subsec:Space_discr}.
We verify the AP property of the fully discretized micro-macro scheme in Section~\ref{subsec:spatial_discretization_AP}.
Extensions of the proposed spatial discretization to the original three-dimensional kinetic equation are discussed later in Section~\ref{subsec:3D}.

\subsection{Time discretization}
\label{subsec:Time_discr}

To discretize \eqref{eq:macro_moment} and \eqref{eq:micro_moment} in time, we assume a uniform time step $\dt$ with time levels $t^n:=n\dt$ and let $\uMe\approx\uM(t^n,\cdot)$ and $\ume\approx\um(t^n,\cdot)$ satisfy 
\begin{subequations}
\begin{align}
\label{eq:macro_time_discre}
\frac{\uMi-\uMe}{\dt} + \vint{\mM \mm^T\Omega} \cdot\nabla_r \umi + \sig{a} \uMi &= 0\:,
\\
\label{eq:micro_time_discre}
\frac{\umi-\ume}{\dt} + \frac{1}{\epsilon}\vint{\mm\mm^T\Omega} \cdot\nabla_r \ume  +  \left(\sig{a}+\frac{\sig{s}}{\epsilon^2}+{\sigF}\Fm\right) \umi 
&=   - \frac{1}{\epsilon^2}\vint{\mm\mM\Omega} \cdot\nabla_r \uMe\:.
\end{align}
\end{subequations}
We rewrite \eqref{eq:micro_time_discre} as 
\begin{equation}
\umi
=  \GammaM \ume - \frac{\dt}{\epsilon} \GammaM \vint{\mm\mm^T\Omega} \cdot\nabla_r \ume - \frac{\dt}{\epsilon^2} \GammaM \vint{\mm\mM\Omega} \cdot\nabla_r \uMe\:,
\label{eq:micro_time_discre_gamma}
\end{equation}
where $\GammaM\in\bbR^{\nm\times\nm}$ is a non-singular, diagonal matrix with elements
\begin{equation}
%\GammaM_{(\ell,k),(\ell,k)} = \frac{\epsilon^2}{\epsilon^2\left(  1  + \sig{a}\dt + {\sigF}\dt\Fm_{(\ell,k),(\ell,k)}\right) + \sig{s}\dt }\:. 
\GammaM_{(\ell,k),(\ell,k)} = \epsilon^2 \left(\epsilon^2(  1  + \sig{a}\dt + {\sigF}\dt\Fm_{(\ell,k),(\ell,k)}) + \sig{s}\dt \right)^{-1}  \in (0,1) \:. 
\label{eq:Gamma_def}
\end{equation} 
To obtain an explicit update for \eqref{eq:macro_time_discre}, we replace the implicit term $\umi$ in \eqref{eq:macro_time_discre} with the right-hand side of \eqref{eq:micro_time_discre_gamma}.
This gives
\begin{subequations}
\label{eq:time_discre}
\begin{equation}
\begin{alignedat}{2}
%(1+\sig{a}\dt)\uMi &= \uMe - \dt\vint{ \mM \mm^T \Omega} \cdot\nabla_r (\GammaM \ume) \\
%&+ \frac{\dt^2}{\epsilon}    \vint{ \mM \mm^T \Omega} \cdot \nabla_r \left( \GammaM \vint{\mm\mm^T\Omega} \cdot \nabla_r \ume \right)\\
%&+ \frac{\dt^2}{\epsilon^2} \vint{ \mM \mm^T \Omega} \cdot \nabla_r \left( \GammaM \vint{\mm\mM\Omega} \cdot \nabla_r \uMe\right)\:.
(1+\sig{a}\dt)\uMi &= \uMe - \dt\vint{ \mM \mm^T \Omega} \cdot\nabla_r (\GammaM \ume) \\
&\quad + \frac{\dt^2}{\epsilon}    \vint{ \mM \mm^T \Omega} \cdot \nabla_r \left( \GammaM \vint{\mm\mm^T\Omega} \cdot \nabla_r \ume \right)\\
&\quad + \frac{\dt^2}{\epsilon^2} \vint{ \mM \mm^T \Omega} \cdot \nabla_r \left( \GammaM \vint{\mm\mM\Omega} \cdot \nabla_r \uMe\right)\:.
\end{alignedat}
\label{eq:macro_time_discre_gamma}
\end{equation}
Since $\sig{a}$ and $\sig{s}$ are functions of $r$, to avoid non-conservative products in \eqref{eq:micro_time_discre_gamma},
we perform a change of variables before and after solving \eqref{eq:micro_time_discre_gamma}.
Specifically, we update $\umi$ by computing
\begin{equation}
\begin{alignedat}{2}
\vme\quad &= \epsilon^2\GammaM^{-1} \ume\:,\\
\vmi
&=  \GammaM \vme - \frac{\dt}{\epsilon} \vint{\mm\mm^T\Omega} \cdot\nabla_r (\GammaM\vme) - \dt \vint{\mm\mM\Omega} \cdot\nabla_r \uMe\:,\\
\umi &= \epsilon^{-2}\GammaM \vmi\:.
\end{alignedat}
\label{eq:micro_time_discre_gamma_v}
\end{equation}
\end{subequations}

\subsection{Reduced linear kinetic equation and the micro-macro system}
\label{subsec:reduced_kinetic_eqn}

For the remainder of the paper, we restrict ourselves to a reduced two-dimensional linear kinetic equation that is valid when $\p_{z} f = 0$: 
\begin{equation}
\epsilon\p_t f + \Ox \p_{x} f + \Oy \p_{y} f = \frac{\sig{s}}{\epsilon}\left(\frac{1}{4\pi}\vint{f}- f\right) - \epsilon\sig{a} f\:.
\label{eq:transport_reduced}
\end{equation} 
In this setting, the micro-macro system \eqref{eq:macro_moment}--\eqref{eq:micro_moment} becomes
\begin{subequations}
\label{eq:moment_reduced}
\begin{align}
\p_t \uM + \,&(\vint{\mM \mm^T\Ox} \p_{x} + \vint{ \mM \mm^T \Oy} \p_{y}) \um + \sig{a} \uM = 0\:,
\label{eq:macro_moment_reduced}
\\
\p_t \um + \,&\frac{1}{\epsilon} (\vint{\mm\mm^T\Ox} \p_{x} + \vint{\mm\mm^T\Oy} \p_{y})\um  + \sig{a} \um \nonumber \\
&\quad= - \frac{\sig{s}}{\epsilon^2} \um  - {\sigF}\bar{F}\um - \frac{1}{\epsilon^2}(\vint{\mm\mM\Ox} \p_{x} + \vint{\mm\mM\Oy} \p_{y})\uM\:.
\label{eq:micro_moment_reduced}
\end{align}
\end{subequations}%
Applying the time discretization in \eqref{eq:time_discre} to \eqref{eq:moment_reduced} leads to the reduced scheme
\begin{subequations}
\begin{align}
(1+\sig{a}\dt)&\uMi = \uMe - {\dt}(\ax^T \p_{x} + \ay^T \p_{y}) (\GammaM \ume)
+ \frac{\dt^2}{\epsilon}
 \nabla_{(x,y)}\cdot\left( \Qm \nabla_{(x,y)} \ume\right) \nonumber\\
&\qquad\qquad\,\,\,+ \frac{\dt^2}{\epsilon^2} \nabla_{(x,y)}\cdot\left( \QM \nabla_{(x,y)} \uMe\right) \:,
 \label{eq:macro_time_discre_2D}\\
&\vme \quad= \epsilon^2\GammaM^{-1} \ume\:,\nonumber\\ 
&\vmi
=  \GammaM \vme - \frac{\dt}{\epsilon} (\Ax \p_x + \Ay \p_y) (\GammaM\vme) - \dt (\ax \p_x + \ay \p_y)\uMe\:, 
\label{eq:micro_time_discre_2D}\\
&\umi = \epsilon^{-2}\GammaM \vmi\:,\nonumber
\end{align}
\end{subequations}
where $\ax:=\vint{\mM\mm\Ox}\in\bbR^{\nm}$, $\ay:=\vint{\mM\mm\Oy}\in\bbR^{\nm}$, $\Ax:=\vint{\mm\mm^T\Ox}\in\bbR^{\nm\times\nm}$, $\Ay:=\vint{\mm\mm^T\Oy}\in\bbR^{\nm\times\nm}$, and $\Qm$ and $\QM$ are $2$-by-$2$ block matrices:
\begin{equation}
\arraycolsep=1.6pt\def\arraystretch{1.5}
\Qm:=\left(\begin{array}{cc}
\tQxx & \tQxy\\
\tQyx & \tQyy
\end{array}\right)
=\left(\begin{array}{cc}
\ax^T\GammaM \Ax & \ax^T\GammaM \Ay\\
\ay^T\GammaM \Ax & \ay^T\GammaM \Ay
\end{array}\right)
=\left(\begin{array}{cc}
\g \axx^T & \g \axy^T\\
\g \ayx^T & \g \ayy^T
\end{array}\right)
%=\left(\begin{array}{cc}
%\gx \vint{ \mM\mm^T\Ox^2} & \gx \vint{ \mM\mm^T\Ox\Oy}\\
%\gy \vint{ \mM\mm^T\Oy\Ox} & \gy \vint{ \mM\mm^T\Oy^2}
%\end{array}\right)\:,
\label{eq:Qm_def}
\end{equation}
and 
\begin{equation}
\arraycolsep=1.6pt\def\arraystretch{1.5}
\QM:=\left(\begin{array}{cc}
\bQxx & \bQxy\\
\bQyx & \bQyy
\end{array}\right)
=\left(\begin{array}{cc}
\ax^T\GammaM \ax & \ax^T\GammaM \ay\\
\ay^T\GammaM \ax & \ay^T\GammaM \ay
\end{array}\right)
%=\left(\begin{array}{cc}
%\gx \vint{ \mM^2\Ox^2} & \gx \vint{ \mM^2\Ox\Oy}\\
%\gy \vint{ \mM^2\Ox\Oy} & \gy \vint{ \mM^2\Oy^2}
%\end{array}\right)
=\left(\begin{array}{cc}
\g \aMxx & \g \aMxy\\
\g \aMyx & \g \aMyy
\end{array}\right)
=\left(\begin{array}{cc}
 \frac{1}{3} \g& 0\\
0 &  \frac{1}{3} \g
\end{array}\right)\:,
\label{eq:QM_def}
\end{equation}
with $\axx:=\vint{\mM\mm\Ox^2}\in\bbR^{\nm}$, $\ayy:=\vint{\mM\mm\Oy^2}\in\bbR^{\nm}$, $\axy=\ayx:=\vint{\mM\mm\Ox\Oy}\in\bbR^{\nm}$, $\aMxx:=\vint{\mM^2\Ox^2}=\frac{1}{3}$, $\aMyy:=\vint{\mM^2\Oy^2}=\frac{1}{3}$, and $\aMxy=\aMyx:=\vint{\mM^2\Ox\Oy}=0$.
The reductions in \eqref{eq:Qm_def} and \eqref{eq:QM_def} follow from direct evaluation 
using the fact that all but one entries of $\ax$ and $\ay$ are zero%
\footnote{This is due to the facts that $\mM$ is a constant, that $\Ox$ and $\Oy$ are scalar multiples of some entries in $\mm$, and that entries of $\bfm$ are orthogonal.}%
, and $\g\in\bbR$ is the diagonal element of $\G$ corresponding to the index of nonzero entry of $\ax$%
\footnote{An equivalent definition of $\g$ is the diagonal element of $\G$ corresponding to the index of nonzero entry of $\ay$. See Appendix~\ref{appendix:diffusion_calculation} for details.}.
The detailed calculation is given in Appendix~\ref{appendix:diffusion_calculation}.

\subsection{Spatial discretization}
\label{subsec:Space_discr}
In this subsection, we introduce the finite difference scheme used to discretize the spatial derivatives in the micro-macro scheme \eqref{eq:time_discre}.
Here we consider $\bbR^2$ as the spatial domain and a uniform mesh on $\bbR^2$ with points $(x_i,y_j)$.
The distances between mesh points in the $x$ and $y$ directions are denoted by $\dx$ and $\dy$, respectively.
We also assume that the aspect ratio $(\dx/\dy)$ of the spatial mesh is bounded from above and away from zero.
We use $w$ and $\bw$ to denote the scalar and vector-valued functions on $\bbR^2$.
For $w\colon \bbR^2\to\bbR$, we denote $w_{i,j}=w(x_i,y_j)$. 
For $\bw\colon \bbR^2\to\bbR^m$, we change the notation and use $\bw_{i,j}$ to denote $\bw(x_i,y_j)$ instead of the entries of $\bw$.
We summarize the spatial discretization for each term in \eqref{eq:time_discre} as follows.

For the advection term in the macro equation \eqref{eq:macro_time_discre_2D}, we use the standard central difference scheme with additional artificial dissipation terms.
For the diffusion terms in \eqref{eq:macro_time_discre_2D}, we adopt the centered symmetric scheme proposed in \cite{Gunter-2005} and modify the scheme by introducing some averaging coefficients into the diffusion stencil.
For the micro equation \eqref{eq:micro_time_discre_2D}, we discretize the micro and macro advection terms with a second-order kinetic upwind scheme and a central difference scheme, respectively.
The artificial dissipation terms and the modified centered symmetric scheme in the discretization of \eqref{eq:macro_time_discre_2D} are needed for proving
the \textit{positive-preserving property} of the scheme.
Specifically, they guarantee that $\uMi\geq0$ on the spatial mesh provided $\uMe\geq0$ on the spatial mesh.

\subsubsection{Macro equation - advection term}
\label{subsubsec:spatial_discretization_hyp}

For the advection term in the macro equation \eqref{eq:macro_time_discre_2D}, we use the central difference scheme with additional artificial dissipation. 
Specifically, the advection term in \eqref{eq:macro_time_discre_2D} is approximated by
\begin{equation}
\left((\ax^T \p_{x} + \ay^T \p_{y}) (\GammaM \ume)\right)_{i,j}
\approx 
(\ax^T \Dcx + \ay^T \Dcy)(\Gc \umec) - \CLxF (\dx^3\ddx +\dy^3 \ddy)\uMec \:,
\label{eq:advection_discretization}
\end{equation}
where $\CLxF$ is the artificial dissipation parameter and $\Dcx$, $\Dcy$ are central difference operators.
For functions $\bw$ on the spatial domain,
\begin{equation}
\Dcx(\wc):=\frac{\we-\ww}{2\dx}\:,
\quand
\Dcy(\wc):=\frac{\wn-\ws}{2\dy}\:.
\label{eq:Dc_def}
\end{equation}
For functions $w = w(x,y)$, we define the artificial dissipation operator in the $x$-direction by 
\begin{equation}
{\ddx}(\wMc):= \frac{1}{{\dx^4}}\left((\wMep - \wMem)-(\wMwp - \wMwm)\right)\:,
\label{eq:ddx_def}
\end{equation}
where 
\begin{equation}
\wMep := \wMe - \frac{\dx}{2} s_{i+1,j}^x\:,\quad
\wMem := \wMc + \frac{\dx}{2} s_{i,j}^x\:,
\label{eq:ddx_edge_values}
\end{equation}
and
\begin{equation}
s_{i,j}^x = \minmod\left\{\theta\frac{ \wMe - \wMc}{\dx},
\frac{\wMe-\wMw}{2\dx}, \theta\frac{\wMc-\wMw}{\dx} \right\}\:.
\label{eq:slope_minmod}
\end{equation}
Here $\theta\in[1,2]$ is a parameter, and the minmod limiter returns the real number with the smallest absolute value in the convex hull of the three arguments (see \cite[Section~16.3]{LeVeque-1992}).
It can be verified that when $\theta=1$, the minmod limiter leads to an inconsistent solution in the diffusion limit (see \cite{McClarren-Lowrie-2008} for relevant discussion in the discontinuous Galerkin setting).
On the other hand, it will be shown in Section~\ref{subsec:pos_proof} that to prove the positive-preserving property, it is required that $\theta<2$. 
Thus, we consider $\theta\in(1,2)$ in the remainder of this paper.
The operator $\ddy$ on the $y$-direction is defined analogously.
In this paper, we choose
\begin{equation}
\CLxF = \Theta \frac{\g_{\max}}{\epsilon}\:,\quad\textup{with}\quad \g_{\max}:=\frac{\epsilon^2}{\epsilon^2+\sig{s}^{\min}\dt}\:,
\label{eq:gamma_max_def}
\end{equation}
where the $\Theta:=\frac{1}{2-\theta}$ and $\sig{s}^{\min}>0$ is the lower bound of $\sig{s}$ on the space.
This choice of $\CLxF$ ensures the positive-preserving property of the scheme. (See Section~\ref{subsec:pos_proof} for details.)

To prove the AP property of the fully discretized scheme (see Section~\ref{subsec:spatial_discretization_AP}), the artificial dissipation needs to vanish as $\epsilon\to0$.
We now show that, under a mild time-step assumption, the choice of $\CLxF$ forces the artificial dissipation to vanish as $\epsilon\to0$.
By Young's inequality,
\begin{equation}
\label{eq:artificial_dissipation_bound}
%\CLxF \dx^3=\frac{\epsilon\dx^3}{\epsilon^2+\sig{s}^{\min}\dt} \leq \frac{\epsilon^\delta}{\epsilon^2+\sig{s}^{\min}\dt} \left(\frac{\epsilon^2}{\frac{2}{1-\delta}}+\frac{\dx^{\frac{6}{1+\delta}}}{\frac{2}{1+\delta}}\right)\:,\quad\forall \delta\in(0,1) \:.
\CLxF \dx^3 = \Theta\frac{\g_{\max}}{\epsilon}\dx^3=
\Theta\frac{\epsilon\dx^3}{\epsilon^2+\sig{s}^{\min}\dt} \leq 
%\Theta\frac{\epsilon^{{1}/{3}}\dx}{\epsilon^2+\sig{s}^{\min}\dt}\left(\frac{\epsilon^2}{3}+\frac{\dx^{3}}{{3}/{2}}\right) \:,
\Theta\epsilon^{{1}/{3}}\dx\frac{\epsilon^2+2\dx^3}{3(\epsilon^2+\sig{s}^{\min}\dt)} \:,
\end{equation}
and a similar upper bound can be obtained for $\CLxF \dy^3$. 
Also, we note that since $\theta>1$, the minmod limiter always returns the center argument when away from the extrema and $\dx$ is sufficiently small. 
In this case, $\ddx$ is a second-order approximation to $-\frac{1}{4}\p_{x}^4$.
Therefore, under a proper regularity assumption on $\uM$, there exists some constant $M>0$ such that $\ddx \uMc\leq M+O(\dx)$ away from the extrema of $\uM$.
A similar upper bounded can be obtained for $\ddy\uMc$.
Therefore, the artificial dissipation vanishes as $\epsilon\to0$ if the upper bound in \eqref{eq:artificial_dissipation_bound} goes to zero as $\epsilon\to0$.
Suppose that the time step $\dt$ satisfies $\dt\geq C(\dx+\dy)^{3}$ for some constant $C>0$, then the upper bound in \eqref{eq:artificial_dissipation_bound} is an $O(\epsilon^{1/3})$ term and thus the artificial dissipation vanishes as $\epsilon\to0$. 
%Thus \eqref{eq:artificial_dissipation_bound} implies that the artificial dissipation vanishes as $\epsilon\to0$ provided that 
This time-step assumption is invoked in Section~\ref{subsec:spatial_discretization_AP} for proving the AP property of the proposed scheme, and it will be justified later in Remark~\ref{remark:consistency}, Section~\ref{subsec:pos_cfl}.
Note that \eqref{eq:artificial_dissipation_bound} also implies that the artificial dissipation goes to zero as $\dx\to0$, which preserves the consistency of the proposed scheme as discussed later in Remark~\ref{remark:consistency}.

\subsubsection{Macro equation - diffusion term}
\label{subsubsec:spatial_discretization_dif}

For the two diffusion terms in \eqref{eq:macro_time_discre_2D},
we apply a modified version of the centered symmetric scheme proposed in \cite{Gunter-2005}, which is formally second-order accurate and conservative.
Specifically, we introduce some averaging coefficients into the discretization of the second derivatives, while keeping the mixed derivative discretizations identical to the ones used in \cite{Gunter-2005}.
At each point $(x_i,y_j)$, we approximate $(\nabla_{(x,y)}\cdot ( \QM \nabla_{(x,y)} \uMe))_{i,j}$ and $(\nabla_{(x,y)}\cdot ( \Qm \nabla_{(x,y)} \ume))_{i,j}$ by
\begin{equation}
\begin{alignedat}{2}
\DD_{\QM}(\uMec) 
=\frac{\aMxx}{\dx^2}  &\sum_{\ell=0,\pm1}c_\ell\left(  \g_{i+\half,j+\frac{\ell}{2}}(\uMe_{i+1,j+\ell} - \uMec) - \g_{i-\half,j+\frac{\ell}{2}}(\uMec - \uMe_{i-1,j+\ell}) \right) \\
+\frac{\aMyy}{\dy^2}  &\sum_{k=0,\pm1} c_k \left(\g_{i+\frac{k}{2},j+\half}(\uMe_{i+k,j+1} - \uMec) - \g_{i+\frac{k}{2},j-\half}(\uMec - \uMe_{i+k,j-1}) \right)
\end{alignedat}
\label{eq:D2_macro_def}
\end{equation}
and 
\begin{equation}
\begin{alignedat}{2}
\DD_{\Qm}(\umec) 
=&\frac{\axx^T}{\dx^2}   \sum_{\ell=0,\pm1} c_\ell \left( \g_{i+\half,j+\frac{\ell}{2}}(\ume_{i+1,j+\ell} - \umec) - \g_{i-\half,j+\frac{\ell}{2}} (\umec - \ume_{i-1,j+\ell}) \right) \\
+&\frac{\ayy^T}{\dy^2}   \sum_{k=0,\pm1} c_k \left(\g_{i+\frac{k}{2},j+\half} (\ume_{i+k,j+1} - \umec) - \g_{i+\frac{k}{2},j-\half}(\umec - \ume_{i+k,j-1}) \right) \\
+&\frac{\axy^T}{2\dx\dy} \sum_{k=\pm1}\left( \g_{i+\frac{k}{2},j+\frac{k}{2}}(\ume_{i+k,j+k} - \umec) - \g_{i+\frac{k}{2},j-\frac{k}{2}}(\ume_{i+k,j-k} - \umec)\right)\:,
\end{alignedat}
\label{eq:D2_micro_def}
\end{equation}
respectively, with averaging coefficients $c_0=\half$ and $c_{\pm1}=\frac{1}{4}$.
Mixed derivatives do not appear in \eqref{eq:D2_macro_def} since $\QM_{xy}=\QM_{yx}=0$ (see \eqref{eq:QM_def}).
Here we compute $\g_{i\pm\half,j}$, $\g_{i,j\pm\half}$ and $\g_{i\pm\half,j\pm\half}$ by taking the harmonic averages
of the adjacent values as proposed in \cite{Sharma-Hammett-2007}.
Specifically,
\begin{equation}
\g_{i+\half,j} := 2\left( \frac{1}{\g_{i+1,j}}+\frac{1}{\g_{i,j}}\right)^{-1}\:,\quad
\g_{i,j+\half} := 2\left( \frac{1}{\g_{i,j+1}}+\frac{1}{\g_{i,j}}\right)^{-1}\:,
\end{equation}
and
\begin{equation}
\g_{i+\half,j+\half} := 4\left( \sum_{k=0,1}\sum_{\ell=0,1}\frac{1}{\g_{i+k,j+\ell}}\right)^{-1}\:.
\end{equation}

\subsubsection{Micro equation}
\label{subsubsec:spatial_discretization_micro}

For the micro equation \eqref{eq:micro_time_discre_2D}, we adopt the spatial discretization used in \cite{Lemou-Mieussens-2008} in the one-dimensional setting, but with second-order discretizations for more accurate solutions. 
Specifically, we discretize the micro advection term by a second-order upwind kinetic scheme (see, for example, \cite{AHT10, GH12, Deshpande-1986,Perthame-1992}):
\begin{equation}
\left((\Ax \p_x + \Ay \p_y) \GammaM\vme\right)_{i,j}
\approx 
(\Ax^+ \Dmx + \Ax^- \Dpx + \Ay^+ \Dmy + \Ay^- \Dpy)(\Gc \vmec) \:,
\end{equation}
where $\Ax^\pm:=\vint{\mm\mm^T\Ox^\pm}$, $\Ay^\pm:=\vint{\mm\mm^T\Oy^\pm}$, with $\Ox^\pm:=\max\{\pm\Ox,0\}$ and $\Oy^\pm:=\max\{\pm\Oy,0\}$.
In the $x$-direction, $\Dpx$ and $\Dmx$ are defined as 
\begin{equation}
\begin{alignedat}{2}
\Dpx(\wc)&:=\frac{1}{\dx}\left(\left(\we - \frac{\wee-\wc}{4} \right) - \left(\wc + \frac{\we-\ww}{4}  \right)\right) \:,\\
\Dmx(\wc)&:=\frac{1}{\dx}\left(\left(\wc - \frac{\we-\ww}{4}  \right) - \left(\ww + \frac{\wc-\www}{4} \right)\right) \:.
%\Dpy(\wc)&:=\frac{1}{\dy}\left(\left(\wn - \frac{\wnn-\wc}{4} \right) - \left(\wc + \frac{\wn-\ws}{4}  \right)\right) \:,\\
%\Dmy(\wc)&:=\frac{1}{\dy}\left(\left(\wc - \frac{\wn-\ws}{4}  \right) - \left(\ws + \frac{\wc-\wss}{4} \right)\right) \:.
\end{alignedat}
\end{equation}
In the $y$-direction, $\Dpy$ and $\Dmy$ are defined similarly.
The macro advection term is discretized  using the central difference operators given in \eqref{eq:Dc_def}, i.e.,
\begin{equation}
\left((\ax \p_{x} + \ay \p_{y}) \uMe\right)_{i,j}
\approx 
(\ax \Dcx + \ay \Dcy)\uMec \:.
\end{equation}

\subsection{Fully discretized micro-macro scheme and the AP property}
\label{subsec:spatial_discretization_AP}

We now show in Theorem~\ref{thm:AP} that under some reasonable assumptions, the fully discretized scheme for the micro-macro system \eqref{eq:moment_reduced} recovers a standard explicit discretization of the diffusion equation
\begin{equation}
\p_t\uM - \p_x \left(\frac{1}{3\sig{s}}\p_x \uM \right)- \p_y \left(\frac{1}{3\sig{s}}\p_y \uM\right) + \sig{a} \uM = 0
\label{eq:diffusion_2D}
\end{equation}
when $\epsilon\to0$.
For reference, the scheme is
\begin{subequations}
\label{eq:full_discre_2D}
\begin{align}
%(1+(\sig{a})_{i,j}\dt)\uMic &= \uMec + \frac{\dt^2}{\epsilon}\DD_{\Qm}(\umec)
%+ \frac{\dt^2}{\epsilon^2} \DD_{\QM}(\uMec)  \label{eq:macro_full_discre_2D}\\
%&\quad- {\dt}\left((\ax^T \Dcx + \ay^T \Dcy)(\Gc \umec) + \frac{\g_{\max}}{\epsilon} (\dx^3\ddx +\dy^3 \ddy)\uMec\right)\:, \nonumber
(1+(\sig{a})_{i,j}\dt)\uMic &= \uMec + \frac{\dt^2}{\epsilon}\DD_{\Qm}(\umec)
+ \frac{\dt^2}{\epsilon^2} \DD_{\QM}(\uMec)  
\label{eq:macro_full_discre_2D}\\
 - {\dt}&\Bigg((\ax^T \Dcx + \ay^T \Dcy)(\Gc \umec) - \Theta \frac{\g_{\max}}{\epsilon} (\dx^3\ddx +\dy^3 \ddy)\uMec\Bigg)
\:, \nonumber\\
\vmec\,\,\, &=\epsilon^2\Gc^{-1}\umec \:, \label{eq:change_var}\\
\vmic
&= \Gc \vmec  - \frac{\dt}{\epsilon} (\Ax^+ \Dmx + \Ax^- \Dpx + \Ay^+ \Dmy + \Ay^- \Dpy)(\Gc \vmec) \nonumber\\
&\quad- \dt (\ax \Dcx + \ay \Dcy)\uMec\:,
\label{eq:micro_full_discre_2D}\\
\umic &=\epsilon^{-2}\Gc \vmic \:. \label{eq:change_var_back}
\end{align}
\end{subequations}

In Theorem~\ref{thm:AP}, we show that the proposed scheme is \textit{weakly} asymptotic preserving (see \cite{Jin-2012} and discussions therein) under a mild time-step restriction.
Specifically, we prove the AP property assuming that (i) the initial condition is sufficiently close to the equilibrium and (ii) the time step satisfies a lower bound that guarantees the artificial dissipation term vanishes as $\epsilon\to0$ (see \eqref{eq:artificial_dissipation_bound}).

\begin{theorem}
\label{thm:AP}
Assume that (i) the initial conditions $\uMc^{0}$ and $\umc^{0}$ are both $O(1)$ quantities and (ii) $\dt\geq C (\dx+\dy)^3$. 
When $\epsilon\to0$, the macro update scheme \eqref{eq:macro_full_discre_2D} becomes a consistent 9-point discretization of the diffusion equation \eqref{eq:diffusion_2D}:
\begin{equation}
\label{eq:macro_limit_2D}
\begin{alignedat}{2}
(1+(\sig{a})_{i,j}\dt)\uMic = \uMec &+ \frac{\dt}{\dx^2} \sum_{\ell=0,\pm1} c_\ell\left( \frac{\uMe_{i+1,j+\ell}-\uMec}{3{(\sig{s})}_{i+\half,j+\frac{\ell}{2}}}
-\frac{\uMec-\uMe_{i-1,j+\ell}}{3{(\sig{s})}_{i-\half,j+\frac{\ell}{2}}}\right)\\
&+ \frac{\dt}{\dy^2} \sum_{k=0,\pm1} c_k\left(\frac{\uMe_{i+k,j+1}-\uMec}{3{(\sig{s})}_{i+\frac{k}{2},j+\half}}-\frac{\uMec-\uMe_{i+k,j-1}}{3{(\sig{s})}_{i+\frac{k}{2},j-\half}}\right)\:,
\end{alignedat}
\end{equation}
where $c_0=\frac{1}{2}$ and $c_{\pm1}=\frac{1}{4}$.
\end{theorem}

\begin{proof}
We first prove by induction that when $\epsilon\to0$, $\uMec$ and $\umec$ are $O(1)$ quantities for all $n\in\bbN$ under assumptions (i) and (ii).
The initial case is given by assumption (i).
Suppose then that $\uMec$ and $\umec$ are $O(1)$ quantities.
From \eqref{eq:Gamma_def}, each element of $\Gc$ is in $(0,1)$, 
\begin{equation}
\label{eq:gamma_bounds}
\epsilon^2\Gc^{-1}=(\sig{s})_{i,j}\dt + O(\epsilon^2)\:,
\quand
\frac{\dt}{\epsilon^p}\Gc\leq\frac{\epsilon^{2-p}}{(\sig{s})_{i,j}}= O(\epsilon^{2-p})\:,\quad p=0,\,1,\,2\:.
\end{equation}
Together \eqref{eq:change_var} and \eqref{eq:gamma_bounds} imply that $\vmec$ is $O(1)$.
It then follows from \eqref{eq:gamma_bounds} that the micro advection term in \eqref{eq:micro_full_discre_2D} is an $O(\epsilon)$ term and thus vanishes as $\epsilon\to0$.
Combining \eqref{eq:change_var}--\eqref{eq:change_var_back} and taking $\epsilon\to0$ then leads to
\begin{equation}
\umic = \lim_{\epsilon\to0} \left(\Gc \umec - \frac{\dt}{\epsilon^2}\Gc (\ax \Dcx + \ay \Dcy)\uMec\right)\:,
\end{equation}
which implies that $\umic$ is an $O(1)$ term since, by \eqref{eq:gamma_bounds}, both terms on the right-hand side are $O(1)$.

On the other hand, it follows from \eqref{eq:gamma_bounds} that the advection term in \eqref{eq:macro_full_discre_2D} is $O(\epsilon^2)$ and that the diffusion term associated to $\Qm$ is $O(\epsilon)$ (see \eqref{eq:Qm_def}).
Thus, these two terms vanish as $\epsilon\to0$. 
In addition, the artificial dissipation in \eqref{eq:macro_full_discre_2D} also vanishes as $\epsilon\to0$ under assumption (ii) (see \eqref{eq:artificial_dissipation_bound}).
Substituting \eqref{eq:D2_macro_def} into \eqref{eq:macro_full_discre_2D} and taking $\epsilon\to0$ then leads to
the nine-point discretization \eqref{eq:macro_limit_2D} that guarantees $\uMic$ is $O(1)$.
Hence, it is proved by induction that for all $n\in\bbN$, $\uMec$ and $\umec$ are both $O(1)$, and thus \eqref{eq:macro_full_discre_2D} becomes the nine-point discretization \eqref{eq:macro_limit_2D} when $\epsilon\to0$.\qed
\end{proof}

\begin{remark}
As discusses in \cite{Larsen-Morel-Miller-1987}, the diffusion limits can be classified into three types based on the scaling of the spatial mesh. 
Specifically, the three types of diffusion limits are the \textup{thick} $((\dx ,\, \dy) = O(1))$, \textup{intermediate} $((\dx,\,\dy) = O(\epsilon))$, and \textup{thin} $((\dx,\,\dy) = O(\epsilon^\ell)$, $\ell\geq 2)$ limits.
We note that the AP and positivity-preserving properties proved in Theorems~\ref{thm:AP} and \ref{thm:positivity} hold for all three types of diffusion limits.
Further, it follows from \eqref{eq:artificial_dissipation_bound} that the artificial dissipation vanishes in the intermediate and thin diffusion limits regardless of $\dt$.
Thus the time-step restriction in assumption (ii) is necessary only when the thick diffusion limit is considered.
\end{remark}

\section{The positive-preserving property}
%\label{subsec:pos_conditions}
\label{sec:Positivity}
%\mytodo{state positivity theorem: the scheme is positivity preserving if the following conditions are satisfied...}
%\mytodo{pf: derive positivity conditions from the spatial discretization}
%\mytodo{emphasize that these conditions are physical}

%\cdh{Is it ever stated clearly what the positive-preserving property is?}
%In this section, we derive the positive-preserving property of the fully discretized macro scheme \eqref{eq:macro_full_discre_2D}.
In Section~\ref{subsec:pos_proof}, we state and prove Theorem~\ref{thm:positivity}, which gives sufficient conditions to yield the positive-preserving property of the fully discretized scheme \eqref{eq:full_discre_2D}.
The approach we use to enforce these conditions in the proposed scheme is presented in Sections~\ref{subsec:limiters} and \ref{subsec:pos_cfl}.
In Section~\ref{subsec:3D}, we discuss the difficulties in extending the proposed scheme and the associated positivity conditions to the three-dimensional case, and we propose a possible approach.

\subsection{Sufficient conditions for preserving positivity}
\label{subsec:pos_proof}

Here we state Theorem~\ref{thm:positivity} on the positive-preserving property of the fully discretized scheme \eqref{eq:full_discre_2D}.
For convenience, in addition to $\uM\in\bbR$ and $\um\in\bbR^{\nm}$, we also use the notation $\uF:=[\uM, \epsilon\um^T]^T\in\bbR^n$ in Theorem~\ref{thm:positivity}. 
The corresponding vectors are defined as $\aFxx:=[\aMxx, \axx^T]^T=[\frac{1}{3}, \axx^T]^T\in\bbR^n$, $\aFxy:=[\aMxy, \axy^T]^T=[0, \axy^T]^T\in\bbR^n$, and $\aFyy:=[\aMyy, \ayy^T]^T=[\frac{1}{3}, \ayy^T]^T\in\bbR^n$, respectively.
Also, we recall the definition of $\g_{\max}$ in \eqref{eq:gamma_max_def}.
%, and define $\g_{\min}:=\min_{i,j} \Gc^{\min}$, where $\Gc^{\min}$ denotes the smallest diagonal element of $\Gc$.
\begin{theorem}
At time $t^n$, suppose that $\uMec\geq0$ and that $\uFec:=[\uMec, \epsilon(\umec)^T]^T$ satisfies the conditions\\
\begin{minipage}{.05\linewidth}
	~
\end{minipage}~
\begin{minipage}{.30\linewidth}
\begin{equation}
\uMec \pm \epsilon\ax^T\umec \geq 0 \:,
\tag{C1}
\label{cond:Tx}
\end{equation}
\end{minipage}~
\begin{minipage}{.05\linewidth}
	~
\end{minipage}~
\begin{minipage}{.30\linewidth}
\begin{equation}
\uMec \pm \epsilon\ay^T\umec \geq 0 \:,
\tag{C2}
\label{cond:Ty}
\end{equation}
\end{minipage}\\
\begin{minipage}{0.05\linewidth}
	~
\end{minipage}~
\begin{minipage}{.30\linewidth}
\begin{equation}
%\uMec \geq \aMxx\uMec + \epsilon\axx^T\umec \geq 0 \:,
\uMec \geq \aFxx^T\uFec \geq 0 \:,
\tag{C3}
\label{cond:Dxx}
\end{equation}
\end{minipage}~
\begin{minipage}{0.05\linewidth}
	~
\end{minipage}~
\begin{minipage}{.30\linewidth}
\begin{equation}
%\uMec \geq \aMyy\uMec + \epsilon\ayy^T\umec \geq 0 \:,
\uMec \geq \aFyy^T\uFec \geq 0 \:,
\tag{C4}
\label{cond:Dyy}
\end{equation}
\end{minipage}\\
\begin{minipage}{0.05\linewidth}
	~
\end{minipage}~
\begin{minipage}{.30\linewidth}
\begin{equation}
 \uMec \pm 2\aFxy^T\uFec \geq 0 \:,
 \tag{C5}
 \label{cond:Dxy1}
\end{equation}
\end{minipage}~
\begin{minipage}{0.05\linewidth}
	~
\end{minipage}~
\begin{minipage}{.51\linewidth}
\begin{equation}
   \left(\frac{\aFxx^T}{\dx^2}   \pm {2}\frac{\aFxy^T}{\dx\dy}   + \frac{\aFyy^T}{\dy^2}  \right) \uFec  \geq0 \:,
 \tag{C6}
 \label{cond:Dxy2}
\end{equation}
\end{minipage}\\
~\\
for each $(x_i,y_j)$ on the spatial mesh.
Further, assume that $\dt$ satisfies
\begin{equation}
1-\g_{\max} \left(\frac{2\Theta\dt}{\epsilon\dx} + \frac{2\Theta\dt}{\epsilon\dy} 
+ \frac{2\dt^2}{\epsilon^2\dx^2} + \frac{\dt^2}{2\epsilon^2\dx\dy} + \frac{2\dt^2}{\epsilon^2\dy^2}\right)\geq0
\tag{C7}
\label{cond:CFL}
\end{equation}
with $\Theta=\frac{1}{2-\theta}$ depends on the minmod parameter $\theta\in(1,2)$ in \eqref{eq:slope_minmod}.
Then the macro scheme \eqref{eq:macro_full_discre_2D} guarantees that $\uMic\geq0$ for each $(x_i,y_j)$ on the spatial mesh.
\label{thm:positivity}
\end{theorem}

\begin{proof}
For simplicity, we write \eqref{eq:macro_full_discre_2D} as
\begin{equation}
(1+(\sig{a})_{i,j}\dt)\uMic=\uMec + \Tx + \Ty + \Diff\:,
\label{eq:macro_full_discre_2D_short}
\end{equation}
where $\Tx$ and $\Ty$ denote the discretized advection terms in the $x$ and $y$ directions, respectively, and $\Diff$ denotes the sum of the two diffusion terms.
Specifically, %with $\CLxF$ in \eqref{eq:gamma_max_def},
\begin{subequations}
\begin{align}
\label{eq:trans_term_x}
\Tx:=&-  \dt\left( \ax^T \Dcx (\Gc \umec) - {\Theta}\frac{\g_{\max}}{\epsilon} \dx^3\ddx (\uMec)\right)\:,\\
\label{eq:trans_term_y}
\Ty:=&-  \dt\left( \ay^T \Dcy (\Gc \umec) - {\Theta}\frac{\g_{\max}}{\epsilon} \dy^3\ddy (\uMec)\right)\:,\\
\label{eq:diff_term}
\Diff:=&\frac{\dt^2}{\epsilon}\DD_{\Qm}(\umec) + \frac{\dt^2}{\epsilon^2} \DD_{\QM}(\uMec)\:.
\end{align}
\end{subequations}
Since $\sig{a}$ is assumed to be nonnegative, we know from \eqref{eq:macro_full_discre_2D_short} that $\uMic\geq0$ if 
%and only if the right-hand side of \eqref{eq:macro_full_discre_2D} is nonnegative, i.e.,
\begin{equation}
\uMec + \Tx + \Ty + \Diff \geq0\:,
\label{eq:goal}
\end{equation}
which we now show.

We first consider the term $\Tx$.
As shown in Appendix~\ref{appendix:minmod_bound}, the operator $\ddx$ defined in \eqref{eq:ddx_def}, satisfies
\begin{equation}
\ddx(\uMec) \geq \frac{1}{2\dx^4} \left(\frac{1}{\Theta} \uMee - 4 \uMec +\frac{1}{\Theta} \uMew\right)\:.
\label{eq:ddx_bound}
\end{equation}
Applying \eqref{eq:ddx_bound} and the definition of $\Dcx$ in \eqref{eq:Dc_def} to \eqref{eq:trans_term_x} leads to
\begin{equation}
\begin{alignedat}{2}
\Tx &\geq
  \frac{\dt}{2\epsilon\dx} \left(\g_{\max} \left( \uMee - 4\Theta \uMec +  \uMew\right) - \epsilon\ax^T (\Ge \umee - \Gw \umew)\right)\\
&\geq  \frac{\dt}{2\epsilon\dx} \Big(\g_{\max} \left( \uMee - 4\Theta \uMec +  \uMew\right) - \epsilon \g_{\max} \left( |\ax^T\umee| + |\ax^T\umew|\right) \Big)\\
%&=  \frac{\g_{\max} \dt}{2\epsilon\dx} \Big( (\uMee-\epsilon|\ax^T  \umee|) + (\uMew-\epsilon|\ax^T \umew|) \Big)
&=  \frac{\g_{\max} \dt}{2\epsilon\dx} \sum_{k=\pm1} \left(  \uMe_{i+k,j}-\epsilon|\ax^T  \ume_{i+k,j}|  \right)
 - \g_{\max}\frac{2\Theta\dt}{\epsilon\dx} \uMec\:.
\end{alignedat}
\label{eq:trans_term_x_bound}
\end{equation}
Here the second inequality follows from two facts: (i) all diagonal entries of $\G$ are bounded from above by $\g_{\max}$ and (ii) all but one entries of $\ax$ are zero, as discussed in Appendix~\ref{appendix:diffusion_calculation}. 
A similar lower bound for $\Ty$ can be obtained analogously.
It follows from \eqref{cond:Tx} that the first term in the lower bound of $\Tx$ is nonnegative.
Similarly, the corresponding term in the lower bound of $\Ty$ is also nonnegative from \eqref{cond:Ty}.
Thus, by plugging the lower bounds of $\Tx$ and $\Ty$ into \eqref{eq:goal}, it suffices to show 
\begin{equation}
\begin{alignedat}{2}
\left(1 - \g_{\max} \left(\frac{2\Theta\dt}{\epsilon\dx} + \frac{2\Theta\dt}{\epsilon\dy}\right)\right)\uMec 
+ \Diff
\geq0\:.
\end{alignedat}
\label{eq:macro_discre_2D_bound}
\end{equation}

We next consider the term $\Diff$. % $\DD_{\Qm}(\umec)$ and $\DD_{\QM}(\uMec)$.
Substituting \eqref{eq:D2_macro_def} and \eqref{eq:D2_micro_def} into \eqref{eq:diff_term} gives
\begin{equation}
\label{eq:macro_discre_2D_bound_1}
\begin{alignedat}{2}
\Diff 
& = \frac{\dt^2}{\epsilon^2\dx^2} \aFxx^T  \sum_{\ell=0,\pm1} c_\ell \left( \g_{i+\half,j+\frac{\ell}{2}}(\uFe_{i+1,j+\ell} - \uFec) - \g_{i-\half,j+\frac{\ell}{2}} (\uFec - \uFe_{i-1,j+\ell}) \right) \\
&+\frac{\dt^2}{\epsilon^2\dy^2} \aFyy^T  \sum_{k=0,\pm1} c_k \left(\g_{i+\frac{k}{2},j+\half} (\uFe_{i+k,j+1} - \uFec) - \g_{i+\frac{k}{2},j-\half}(\uFec - \uFe_{i+k,j-1}) \right) \\
&+\frac{\dt^2}{2\epsilon^2\dx\dy} \aFxy^T\sum_{k=\pm1}\left( \g_{i+\frac{k}{2},j+\frac{k}{2}}(\uFe_{i+k,j+k} - \uFec) - \g_{i+\frac{k}{2},j-\frac{k}{2}}(\uFe_{i+k,j-k} - \uFec)\right)\:,
\end{alignedat}
\end{equation}
where $\uF=[\uM, \epsilon\um^T]^T$ and the vectors $\aFxx$, $\aFxy$, and $\aFyy$ are defined in the beginning of this section.
%\begin{equation}
%\begin{alignedat}{2}
%\frac{\dt^2}{\epsilon^2} \DD_{\QM}(\uMec) 
%=\frac{\dt^2}{\epsilon^2\dx^2} \aMxx &\sum_{\ell=0,\pm1}c_\ell\left(  \g_{i+\half,j+\frac{\ell}{2}}(\uMe_{i+1,j+\ell} - \uMec) - \g_{i-\half,j+\frac{\ell}{2}}(\uMec - \uMe_{i-1,j+\ell}) \right) \\
%+\frac{\dt^2}{\epsilon^2\dy^2} \aMyy &\sum_{k=0,\pm1} c_k \left(\g_{i+\frac{k}{2},j+\half}(\uMe_{i+k,j+1} - \uMec) - \g_{i+\frac{k}{2},j-\half}(\uMec - \uMe_{i+k,j-1}) \right)
%\end{alignedat}
%\label{eq:diff_macro}
%\end{equation}
%and
%\begin{equation}
%\begin{alignedat}{2}
%\frac{\dt^2}{\epsilon}\DD_{\Qm}(\umec) 
%=&\frac{\dt^2}{\epsilon\dx^2} \axx^T  \sum_{\ell=0,\pm1} c_\ell \left( \g_{i+\half,j+\frac{\ell}{2}}(\ume_{i+1,j+\ell} - \umec) - \g_{i-\half,j+\frac{\ell}{2}} (\umec - \ume_{i-1,j+\ell}) \right) \\
%+&\frac{\dt^2}{\epsilon\dy^2} \ayy^T  \sum_{k=0,\pm1} c_k \left(\g_{i+\frac{k}{2},j+\half} (\ume_{i+k,j+1} - \umec) - \g_{i+\frac{k}{2},j-\half}(\umec - \ume_{i+k,j-1}) \right) \\
%+&\frac{\dt^2}{2\epsilon\dx\dy} \axy^T\sum_{k=\pm1}\left( \g_{i+\frac{k}{2},j+\frac{k}{2}}(\ume_{i+k,j+k} - \umec) - \g_{i+\frac{k}{2},j-\frac{k}{2}}(\ume_{i+k,j-k} - \umec)\right)
%\end{alignedat}
%\label{eq:diff_micro}
%\end{equation}
%with $c_0=\half$ and $c_{\pm1}=\frac{1}{4}$.
%By summing up \eqref{eq:diff_macro}--\eqref{eq:diff_micro}, collecting the terms, and expressing the results in terms of $\uF=[\uM, \epsilon\um^T]^T$, we have 
Collecting the terms in \eqref{eq:macro_discre_2D_bound_1} based on the spatial indices leads to
\begin{equation}
\label{eq:macro_discre_2D_bound_2}
\begin{alignedat}{2}
\Diff\geq
&-\frac{\g_{\max}}{\epsilon^2}\left(\frac{2\dt^2}{\dx^2} \left|\aFxx^T\uFec \right|
+ \frac{\dt^2}{\dx\dy} \left|\aFxy^T\uFec\right| + \frac{2\dt^2}{\dy^2} \left|\aFyy^T\uFec \right| \right)\\
&+ \frac{\dt^2}{2\epsilon^2\dx^2} \sum_{k=\pm1}\g_{i+\frac{k}{2},j} \aFxx^T {\uF}^n_{i+k,j}
+ \frac{\dt^2}{2\epsilon^2\dy^2} \sum_{\ell=\pm1}\g_{i,j+\frac{\ell}{2}} \aFyy^T {\uF}^n_{i,j+\ell} \\
&+ \frac{\dt^2}{4\epsilon^2}\left(\frac{\aFxx^T}{\dx^2} + 2\frac{\aFxy^T}{\dx\dy} + \frac{\aFyy^T}{\dy^2}\right)\sum_{k=\pm1}\g_{i+\frac{k}{2},j+\frac{k}{2}}  {\uF}^n_{i+k,j+k}\\
&+ \frac{\dt^2}{4\epsilon^2}\left(\frac{\aFxx^T}{\dx^2} - 2\frac{\aFxy^T}{\dx\dy} + \frac{\aFyy^T}{\dy^2}\right)\sum_{k=\pm1}\g_{i+\frac{k}{2},j-\frac{k}{2}}  {\uF}^n_{i+k,j-k}\:,
\end{alignedat}
\end{equation}
where the inequality follows from dropping nonnegative terms, taking absolute values, and bounding all $\g$'s with $\g_{\max}$ in the first term. 
From \eqref{cond:Dxx}, \eqref{cond:Dyy}, and \eqref{cond:Dxy2}, all but the first term on the right-hand side of \eqref{eq:macro_discre_2D_bound_2} are nonnegative and thus can be dropped from the inequality.
We then apply \eqref{cond:Dxx}, \eqref{cond:Dyy}, and \eqref{cond:Dxy1} on the remaining term, which yields
\begin{equation}
\Diff\geq -\g_{\max}\left(\frac{2\dt^2}{\epsilon^2\dx^2} + \frac{\dt^2}{2\epsilon^2\dx\dy} +  \frac{2\dt^2}{\epsilon^2\dy^2}\right) \uMec \:.
\label{eq:macro_discre_2D_bound_3}
\end{equation}
After plugging this lower bound into \eqref{eq:macro_discre_2D_bound}, it now suffices to show that
\begin{equation}
\left(1 - \g_{\max} \left(\frac{2\Theta\dt}{\epsilon\dx} + \frac{2\Theta\dt}{\epsilon\dy} 
+ \frac{2\dt^2}{\epsilon^2\dx^2} + \frac{\dt^2}{2\epsilon^2\dx\dy}+  \frac{2\dt^2}{\epsilon^2\dy^2}\right) \right)\uMec\geq0\:.
\label{eq:macro_discre_2D_bound_4}
\end{equation}
From \eqref{cond:CFL}, \eqref{eq:macro_discre_2D_bound_4} holds and the proof is complete.
\end{proof}

Theorem~\ref{thm:positivity} provides sufficient conditions \eqref{cond:Tx}--\eqref{cond:CFL} to preserve positivity of the particle concentrations with the proposed scheme.
In general, these conditions are not readily satisfied.
In Section~\ref{subsec:limiters}, we review two {\limname} limiters proposed in \cite{LH-2016} and adopt them to enforce Conditions~\eqref{cond:Tx}--\eqref{cond:Dxy2}.
For Condition~\eqref{cond:CFL}, we show in Section~\ref{subsec:pos_cfl} that this condition is satisfied if a CFL-type time-step restriction is imposed.

\subsection{{\Limname} limiters}
\label{subsec:limiters}

%\cdh{"Positivity limiters" $\rightarrow$ "Realizability limiters"?}
%The Conditions~\eqref{cond:Tx}--\eqref{cond:Dxy2} are all \emph{linear inequalities} in $\uF:=[\uM,\epsilon\um^T]^T$.
%\cdh{why is linearity important?  This is not clear.}   
%More importantly, they are all \emph{physical}, i.e., for a given $\uF\in\bbR^n$, if the spectral expansion $h(\Omega):=\bfm^T\uF$ is nonnegative on $\bbS^2$, then $\uF$ satisfies these conditions.
The Conditions~\eqref{cond:Tx}--\eqref{cond:Dxy2} are all \emph{physical}, i.e., for a given $\uF\in\bbR^n$, if the spectral expansion $h(\Omega):=\bfm^T\uF$ is nonnegative on $\bbS^2$, then $\uF=:[\uM,\epsilon\um^T]^T$ satisfies these conditions.
Thus, enforcing these physical conditions does not affect the spectral expansions that are already nonnegative.

For a given angular spectral expansion with nonnegative mean, the {\limname} limiters considered in \cite{LH-2016} give approximations that are nonnegative pointwisely on a specific quadrature set while preserving the mean. 
We refer to these limiters as pointwise limiters in this paper.
It is straightforward to verify that the pointwise nonnegativity condition required by these limiters is stronger than Conditions~\eqref{cond:Tx}--\eqref{cond:Dxy2}.
%\change{Further, the pointwise nonnegativity condition and Conditions~\eqref{cond:Tx}--\eqref{cond:Dxy2} are all formulated as \emph{linear inequalities} in $\uF$.}
Thus we borrow the idea of these pointwise limiters, but relax them to enforce \eqref{cond:Tx}--\eqref{cond:Dxy2}.  
These new, relaxed limiters preserve the values of the macro coefficients $\uM$ and modify the micro coefficients $\um$ to satisfy \eqref{cond:Tx}--\eqref{cond:Dxy2}.
We expect these relaxed limiters to be more efficient than the pointwise limiters in terms of accuracy and computation cost, since the relaxed limiters are less likely to be active and they enforces weaker conditions.

We first consider the linear scaling ({\ls}) limiter \cite{Liu-Osher-1996,Zhang-Shu-2010,Zhang-Shu-2011}, which damps the micro coefficients uniformly until some desirable condition $\cond$ on $\uF$ is satisfied.
Specifically, given $\uF=[\uM,\epsilon\um^T]^T$ with $\uM\geq0$, the {\ls} limiter produces an approximation $\uFls:=[\uMls,\epsilon\umls^T]^T$ such that $\uMls=\uM$ and
%with $\uMls=\uM$ and $\umls=\alpha_{\lsnm}\um$, where the scaling ratio $\alpha_{\lsnm}:=\max\{\alpha>0\colon [\uM,\epsilon\alpha\um^T]^T \text{ satisfies } \eqref{cond:Tx}\text{--}\eqref{cond:Dxy}\}$.
\begin{equation}
%\umls=\alpha_{\lsnm}\um \text{ with } \alpha_{\lsnm}:=\argmax_{\alpha\in[0,1]}\left\{\alpha\colon[\uM,\epsilon\alpha\um^T]^T \text{ satisfies } \eqref{cond:Tx}\text{--}\eqref{cond:Dxy2}\right\}\:.
\umls=\alpha_{\lsnm}\um \text{ with } \alpha_{\lsnm}:=\argmax_{\alpha\in[0,1]}\left\{\alpha\colon[\uM,\epsilon\alpha\um^T]^T \text{ satisfies } {\cond}\right\}\:.
\label{eq:ls_limiter}
\end{equation}
The second limiter considered in this paper is the optimization-based ({\opt}) limiter \cite{LHMOT-2016}, which finds the best approximation to the micro coefficients, in the $\ell^2$ sense, that still satisfies $\cond$.
Specifically, given $\uF:=[\uM,\epsilon\um^T]^T$ with $\uM\geq0$, the {\opt} limiter gives an approximation $\uFopt:=[\uMopt,\epsilon\umopt^T]^T$ such that $\uMopt=\uM$ and 
\begin{equation}
%\umopt=\argmin_{\vm\in \bbR^{\nm}}\left\{\frac{1}{2}\|\vm-\um\|_2^2 \colon [\uM,\epsilon\vm^T]^T \text{ satisfies } \eqref{cond:Tx}\text{--}\eqref{cond:Dxy2}\right\}\:.
\umopt=\argmin_{\vm\in \bbR^{\nm}}\left\{\frac{1}{2}\|\vm-\um\|_2^2 \colon [\uM,\epsilon\vm^T]^T \text{ satisfies } {\cond} \right\}\:.
\label{eq:opt_limiter}
\end{equation}

% with $\uMopt=\uM$ and $\umopt$ solves $\minimize_{\vm\in \bbR^{\nm}} \frac{1}{2}\|\vm-\um\|^2$ subject to $[\uM,\epsilon\vm^T]^T \text{ satisfies } \eqref{cond:Tx}\text{--}\eqref{cond:Dxy}$.

Here we apply these limiters on $\uFec$ at each $((x_i,y_j), t^n)$ on the space-time mesh in order to enforce conditions \eqref{cond:Tx}--\eqref{cond:Dxy2}.
We embed the limiters into the proposed scheme such that, when $\uFec$ violates any of \eqref{cond:Tx}--\eqref{cond:Dxy2}, we compute the limited coefficients $\uFecls$ or $\uFecopt$ and then proceed with $\uFec$ replaced by $\uFecls$ or $\uFecopt$.
Since \eqref{cond:Tx}--\eqref{cond:Dxy2} are relaxed from the pointwise nonnegativity condition, we refer to these relaxed limiters as {\lsrlx} and {\optrlx}, respectively.

%We embed the limiters into the proposed scheme such that for each $((x_i,y_j), t^n)$ on the space-time mesh, when the original coefficients $\uFec$ violate any of \eqref{cond:Tx}--\eqref{cond:Dxy2}, we compute the limited coefficients $\uFecls$ or $\uFecopt$ and proceed with $\uFec$ replaced by $\uFecls$ or $\uFecopt$.

In the numerical experiments reported in Section~\ref{sec:num_results}, we compare the {\lsrlx} and {\optrlx} limiters as well as their pointwise versions, denoted respectively as {\lspw} and {\optpw}, considered in \cite{LH-2016}.
The {\lspw} and {\optpw} limiters are formulated by replacing $\cond$ in \eqref{eq:ls_limiter} and \eqref{eq:opt_limiter}, respectively, by the analogous pointwise condition.
From the numerical results in Section~\ref{sec:num_results}, we confirm that the {\lsrlx} and {\optrlx} limiters are more efficient than the {\lspw} and {\optpw} limiters.

\subsection{Positivity time-step restriction}
\label{subsec:pos_cfl}
In this section, we show that Condition~\eqref{cond:CFL} in Theorem~\ref{thm:positivity} is satisfied under a CFL-type time-step restriction stated in the following lemma.

\begin{lemma}
\label{lemma:Pos_CFL}
Condition \eqref{cond:CFL} holds if $\dt$ satisfies
%\begin{equation}
%\dt\leq 
%\max\left\{\left(\frac{\sqrt{5}-2}{2}\right) \epsilon \left(\frac{\dx\dy(\dx+\dy)}{\dx^2+\dy^2}\right),\,
% \sig{s}^{\min} \left(\frac{\dx^2\dy^2}{4(\dx+\dy)^2+2(\dx^2+\dy^2)}\right)
%\right\}
%\end{equation}
\begin{equation}
\dt\leq \max\left\{\dt_{\textup{hyp}},\,\dt_{\textup{par}}\right\}\:,
\label{eq:Pos_CFL}
\end{equation}
where
\begin{equation}
\dt_{\textup{hyp}} := 2\left(\sqrt{\frac{9}{8}+\Theta^2}-\Theta\right) \epsilon \left(\frac{\dx\dy(\dx+\dy)}{4\dx^2+\dx\dy+4\dy^2}\right)\:,
\end{equation}
\begin{equation}
\dt_{\textup{par}} := \sig{s}^{\min}  \left(\frac{\dx^2\dy^2}{(2+\Theta^2)\dx^2+(\half+2\Theta^2)\dx\dy+(2+\Theta^2)\dy^2}\right)\:,
\end{equation}
and the constant $\Theta=\frac{1}{2-\theta}$ depends on the minmod parameter $\theta\in(1,2)$ in \eqref{eq:slope_minmod}.
\end{lemma}

\begin{proof}
From the definition of $\g_{\max}$ in \eqref{eq:gamma_max_def}, \eqref{cond:CFL} is equivalent to
\begin{equation}
h_{\dt}(\epsilon):= \epsilon^2+\sig{s}^{\min}\dt
-\dt\left(\frac{2\Theta\epsilon}{\dx}+\frac{2\Theta\epsilon}{\dy}+\frac{2\dt}{\dx^2}+ \frac{\dt}{2\dx\dy} + \frac{2\dt}{\dy^2}\right)\geq0\:.
\label{eq:h_dt_def}
\end{equation}
%Since $h_{\dt}^{\prime\prime}(\epsilon)=2>0$, the minimizer $\epsilon^*$ of the parabola $h_{\dt}$ satisfies that
%\begin{equation}
%h_{\dt}^\prime(\epsilon^*)=2\epsilon^*
%-2\dt\frac{\dx+\dy}{\dx\dy}=0\:,\textup{ i.e., }
%\epsilon^*=\frac{\dt(\dx+\dy)}{\dx\dy}\:.
%\end{equation}
The minimizer of the parabola $h_{\dt}$ is given by
$\epsilon^*=\frac{\Theta\dt(\dx+\dy)}{\dx\dy}\:.$
Therefore, to prove \eqref{cond:CFL}, it suffices to show that
\begin{equation}
h_{\dt}(\epsilon^*)=\sig{s}^{\min}\dt - \frac{\dt^2}{\dx^2\dy^2}
\left(\Theta^2{(\dx+\dy)^2}+(2\dx^2+\half\dx\dy+2\dy^2)\right)\geq0\:,
\end{equation}
which leads to
\begin{equation}
\dt \leq \sig{s}^{\min} \left(\frac{\dx^2\dy^2}{(2+\Theta^2)\dx^2+(\half+2\Theta^2)\dx\dy+(2+\Theta^2)\dy^2}\right) = \dt_{\textup{par}}\:.
\label{eq:CFL_par}
\end{equation}
On the other hand, since $\sig{s}^{\min}>0$, it follows from \eqref{eq:h_dt_def} that 
\begin{equation}
\epsilon^2-2\dt\left(\frac{\Theta\epsilon}{\dx}+\frac{\Theta\epsilon}{\dy}+\frac{\dt}{\dx^2}+ \frac{\dt}{4\dx\dy} + \frac{\dt}{\dy^2}\right)\geq0
\label{eq:H_dt_def}
\end{equation}
is also a sufficient condition for \eqref{cond:CFL}. 
Since the left-hand side of \eqref{eq:H_dt_def} is quadratic in $\dt$, we obtain
\begin{equation}
\dt \leq 2\left(\sqrt{\frac{9}{8}+\Theta^2}-\Theta\right) \epsilon \left(\frac{\dx\dy(\dx+\dy)}{4\dx^2+\dx\dy+4\dy^2}\right) = \dt_{\textup{hyp}}
\label{eq:CFL_hyp}
\end{equation}
by solving \eqref{eq:H_dt_def} and applying the inequality 
\begin{equation}
(2+\Theta^2)\dx^2+(\half+2\Theta^2)\dx\dy+(2+\Theta^2)\dy^2\geq \left(\frac{9}{8}+\Theta^2\right)(\dx+\dy)^2\:.
\end{equation}
%Rewriting \eqref{eq:H_dt_def} in the quadratic form in $\dt$ gives
%\begin{equation}
%\left(\frac{4\dx^2+\dx\dy+4\dy^2}{2\dx^2\dy^2}\right)\dt^2 + \left(\frac{2\epsilon(\dx+\dy)}{\dx\dy}\right)\dt -\epsilon^2\leq0\:.
%\end{equation}
%By solving this quadratic polynomial and using the inequality $3\dx^2+\frac{5}{2}\dx\dy+3\dy^2\geq \frac{17}{8}(\dx+\dy)^2$, we obtain
%\begin{equation}
%\dt \leq \left(\sqrt{\frac{17}{2}}-2\right) \epsilon \left(\frac{\dx\dy(\dx+\dy)}{4\dx^2+\dx\dy+4\dy^2}\right) = \dt_{\textup{hyp}}\:.
%\label{eq:CFL_hyp}
%\end{equation}
Since \eqref{eq:CFL_par} and \eqref{eq:CFL_hyp} are sufficient conditions for \eqref{cond:CFL}, the claim is proved.
\end{proof}

%\cdh{Check the two changes I made in this paragraph.} 
Since the aspect ratio $(\dx/\dy)$ is assumed to be bounded from above and away from zero, the time-step restriction \eqref{eq:Pos_CFL} switches from a hyperbolic CFL condition to a parabolic CFL condition as $\epsilon\to0$. 
Specifically, when $\epsilon\gg (\dx+\dy)$, \eqref{eq:Pos_CFL} takes the form of a hyperbolic CFL condition, i.e., $\dt\leq C\epsilon(\dx+\dy)$. 
On the other hand, when $\epsilon \ll (\dx+\dy)$, \eqref{eq:Pos_CFL} switches to a parabolic CFL condition, i.e., $\dt\leq C \dx\dy$.
The switch between time-step restrictions is desirable for AP schemes, since the hyperbolic CFL condition becomes prohibitive as $\epsilon\to0$.

\begin{remark}
\label{remark:consistency}
The time-step restriction \eqref{eq:Pos_CFL} justifies the time-step assumption $\dt\geq C(\dx+\dy)^3$, which is made in Section~\ref{subsubsec:spatial_discretization_hyp} and is invoked in the AP property analysis in Section~\ref{subsec:spatial_discretization_AP}.
In addition, if the time step $\dt$ is chosen to be the largest step allowed by \eqref{eq:Pos_CFL}, then \eqref{eq:artificial_dissipation_bound} can be rewritten as 
\begin{equation}
\label{eq:artificial_dissipation_consistency}
\CLxF \dx^3 = \Theta\frac{\g_{\max}}{\epsilon}\dx^3=\Theta
\frac{\epsilon\dx^3}{\epsilon^2+\sig{s}^{\min}\dt} \leq 
%\Theta\frac{\dx^2}{\epsilon^2+\sig{s}^{\min}\dt} \left(\frac{\epsilon^2}{2}+\frac{\dx^{2}}{2}\right) \:,
\Theta\dx^2\frac{{\epsilon^2}+{\dx^{2}}}{2(\epsilon^2+\sig{s}^{\min}\dt)} \leq
C_x \dx^2  \:,
\end{equation}
with constant $C_x$ independent of $\epsilon$ and $\dx$.
Here the first inequality follows from Young's inequality and the second inequality uses the fact that $\dt$ is the largest step allowed by \eqref{eq:Pos_CFL}.
\eqref{eq:artificial_dissipation_consistency} implies that the artificial dissipation vanishes as $\dx\to0$.
Thus, we conclude that the proposed scheme is consistent and AP under \eqref{eq:Pos_CFL}.
\end{remark}

\subsection{Extension to three dimensions}
\label{subsec:3D}

The spatial discretization and positivity conditions proposed in Sections~\ref{subsec:Space_discr} and \ref{subsec:pos_proof} are for the micro-macro system corresponding to the reduced linear kinetic equation \eqref{eq:transport_reduced} in two dimensions introduced in Section~\ref{subsec:reduced_kinetic_eqn}.
For the original three-dimensional kinetic equation \eqref{eq:transport_scaled}, the proposed spatial discretization can be extended naively to obtain a positive-preserving AP scheme, under an extended version of positivity conditions.
However, the straightforward extension results in a discretization that is defined on alternating spatial grids in the diffusion limit. 
The alternating grid comes from the fact that such extension of diffusion stencils \eqref{eq:D2_macro_def} and \eqref{eq:D2_micro_def} calculates the mixed derivatives with the ``edge" values on the three-dimensional stencil.
Specifically, to compute the mixed derivatives at $(x_i,y_j,z_k)$, such extension uses function values at edge points $(x_{i\pm1},y_{j\pm1},z_k)$, $(x_{i\pm1},y_j,z_{k\pm1})$, $(x_i,y_{j\pm1},z_{k\pm1})$.
Meanwhile, to form physical positivity conditions at the edge points, the averaging parameters in the second derivative discretization need to be such that $c_{\pm1}=\frac{1}{4}$ and $c_{0}=0$, where $c_{\pm1}$ denotes the averaging weight on the edge points and $c_{0}$ denotes the averaging weight on the ``face" points $(x_{i\pm1},y_j,z_k)$, $(x_{i},y_{j\pm1},z_k)$, $(x_i,y_{j},z_{k\pm1})$.
The zero weight on the face points makes the extended diffusion discretization a 13-point stencil, including the center and edge points. 
For any two adjacent points, the two diffusion stencils are completely disjoint. 
Thus the resulting discretization is on alternating spatial grids, which may lead to oscillatory solutions.
See, for example, \cite{Haack-Hauck-2008} for relevant discussions.

A possible approach to avoid the occurrence of alternating grids is to modify the extended diffusion stencil such that the mixed derivatives are computed with the values at ``corner" points  $(x_{i\pm1},y_{j\pm1},z_{k\pm1})$.
In this case, the averaging process in the second derivative discretization is performed on the face and corner points, instead of the face and edge points. 
In the case of constant scattering cross-sections, this procedure leads to a 15-point stencil that does not suffer from the problem of alternating grids and is expected to preserve positivity under physical positivity conditions.
The extension to problems with general scattering cross-sections is in the scope of future work.
%Our preliminary result shows that such extension results in a 15-point diffusion stencil, including the center, face, and corner points. 
%The resulting 15-point stencil does not suffer from the problem of alternating grids and is expected to preserve positivity under physical positivity conditions.

\section{Numerical results}
\label{sec:num_results}

In this section, we test the performance of the scheme in solving the reduced kinetic equation \eqref{eq:transport_reduced} for two benchmark problems. We also run a space-time accuracy test.

\subsection{Line source problem}
\label{subsec:linesource}

The line source problem and its semi-analytic solution were first considered in \cite{ganapol-1999}.
The problem has served as a performance benchmark in studying various numerical schemes for solving linear kinetic equations \cite{Brunner-2002,GH12,Hauck-McClarren-2009,Hauck-McClarren-2010,Radice-2013}. It involves an isotropic initial condition supported at the origin of the spatial domain.
In our numerical simulations, the initial condition is approximated by a steep Gaussian distribution centered at the origin with variance $\varsigma^2=9\times 10^{-4}$, i.e.,
\begin{equation}
f^{\rm{in}}(r,\Omega) \approx \frac{1}{4\pi}
\left(\frac{1}{{2 \pi  \varsigma^2}} 
e^{\frac{-(x^2 + y^2)}{2 \varsigma^2}}\right) \:,
\label{eq:Gaussian_ic}
\end{equation} 
and the cross-sections are chosen to be $\sig{t}=\sig{s} = 1.0$.  Tests are run in the kinetic regime ($\epsilon=1$) and the diffusive regime ($\epsilon=10^{-3}$).

The simulation is performed on a truncated spatial domain: a $3 \times 3$  square centered at the origin with zero boundary condition.  The final time is $t_{\rm{final}} = 1.0$ when $\epsilon=1$ and $t_{\rm{final}} = 0.1 $ when $\epsilon=10^{-3}$.
We choose the angular spectral approximation order to be $N=11$ for the kinetic tests and $N=3$ for the diffusive tests.
We perform the computation on a $150\times150$ uniform square spatial mesh with the time step chosen as 0.9 times the maximum value allowed by condition \eqref{eq:Pos_CFL}.
The filter function $\kappa$ is given by $ \kappa(\lambda)=\frac{1}{1+\lambda^4}$ with the filtering parameter ${\sigF}=56.2$.
The minmod parameter in \eqref{eq:slope_minmod} is chosen to be $\theta=1.5$.
In each regime, we solve the problem using the proposed AP scheme with the {\lsrlx} and {\optrlx} limiters and, for comparison, the pointwise {\lspw} and {\optpw} limiters considered in \cite{LH-2016}. 
See Section~\ref{subsec:limiters} for the details of these {\limname} limiters.
We implement the {\lsrlx} and {\lspw} limiters by solving the maximization problems via direct evaluation, since there is no optimization required. 
On the other hand, the minimization problems for the {\optrlx} and {\optpw} limiters are solved respectively using the alternating direction method of multipliers (ADMM) \cite{Boyd-Parikh-2011} and the constraint-reduced Mehrotra-predictor-corrector method (CR-MPC) \cite{LA-2017}, with tolerance $10^{-6}$.
The optimization algorithms are chosen such that the computational cost is minimized.

In the kinetic regime, we use the semi-analytic solution given in \cite{ganapol-1999} as the reference solution.
In the diffusive regime, the reference solution is generated by solving the diffusion equation. \eqref{eq:diffusion_2D} with the explicit 9-point finite difference scheme \eqref{eq:macro_limit_2D}.
In Figure~\ref{fig:linesource_ref}, we plot the two-dimensional heat maps and one-dimensional line-outs (along the $x$-axis) of the particle concentration $\rho=\vint{f}$ in the reference solutions.

\begin{figure}[h]
   \captionsetup[subfigure]{justification=centering}
\centering
\subfloat[$\epsilon=1$, $t=1.0$]
{  \includegraphics[width=0.23\linewidth]{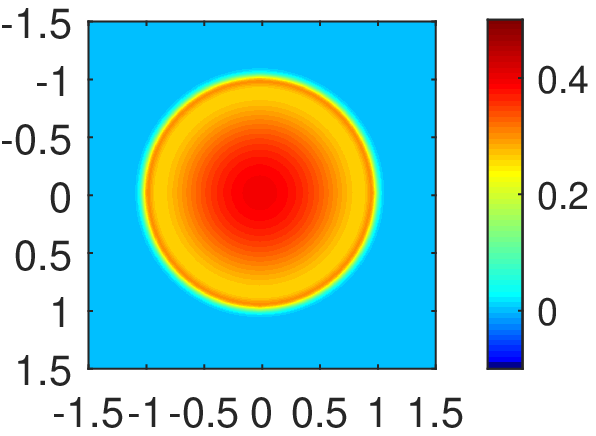}
  \label{fig:kinetic_linesource_2D}~
  \includegraphics[width=0.21\linewidth]{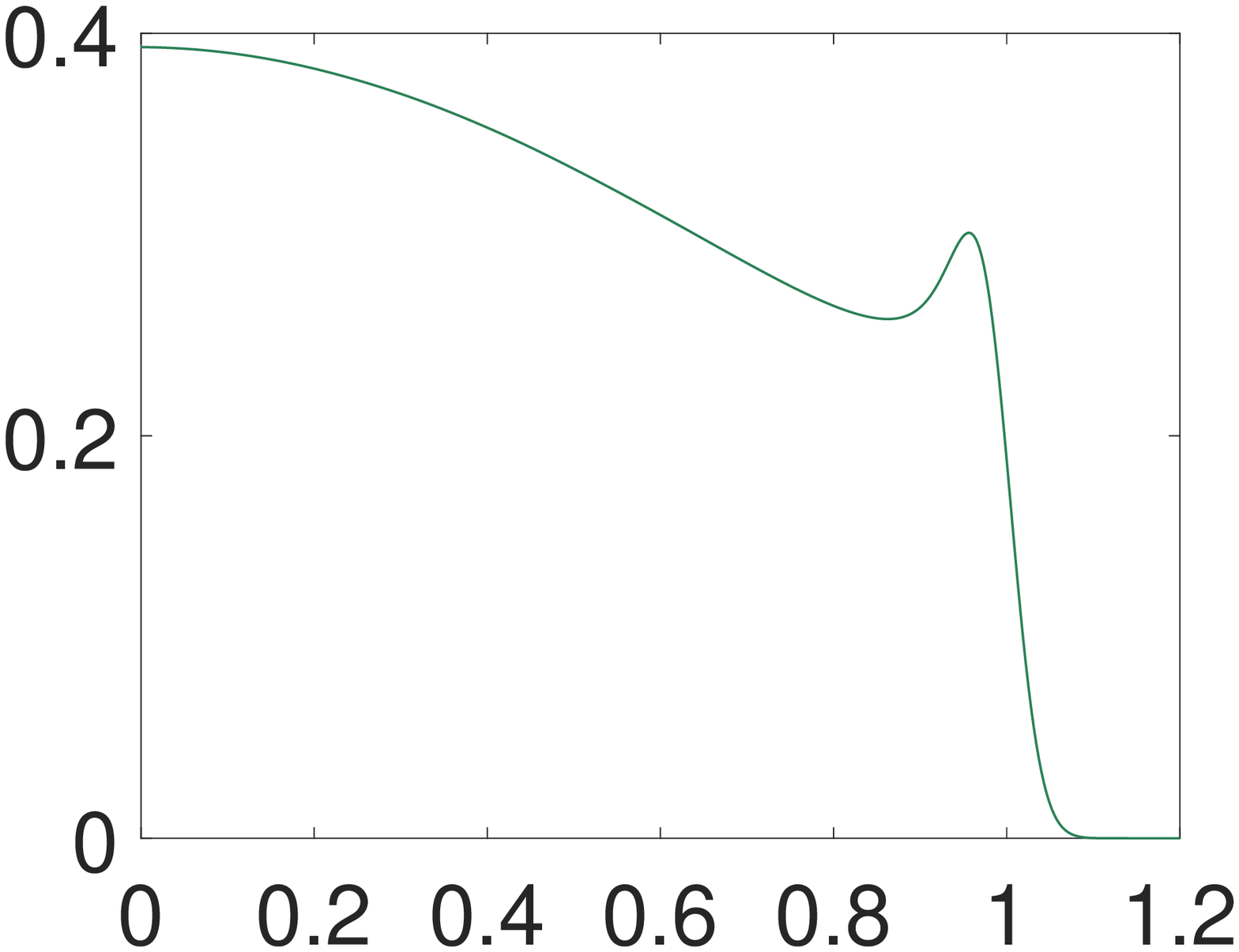}
  \label{fig:kinetic_linesource_LO}}~~
\subfloat[$\epsilon=10^{-3}$, $t=0.1$]
{\includegraphics[width=0.23\linewidth]{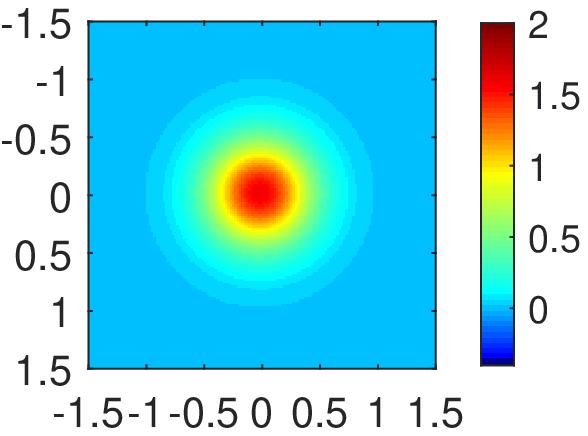}
  \label{fig:diffusion_linesource_2D}~
  \includegraphics[width=0.21\linewidth]{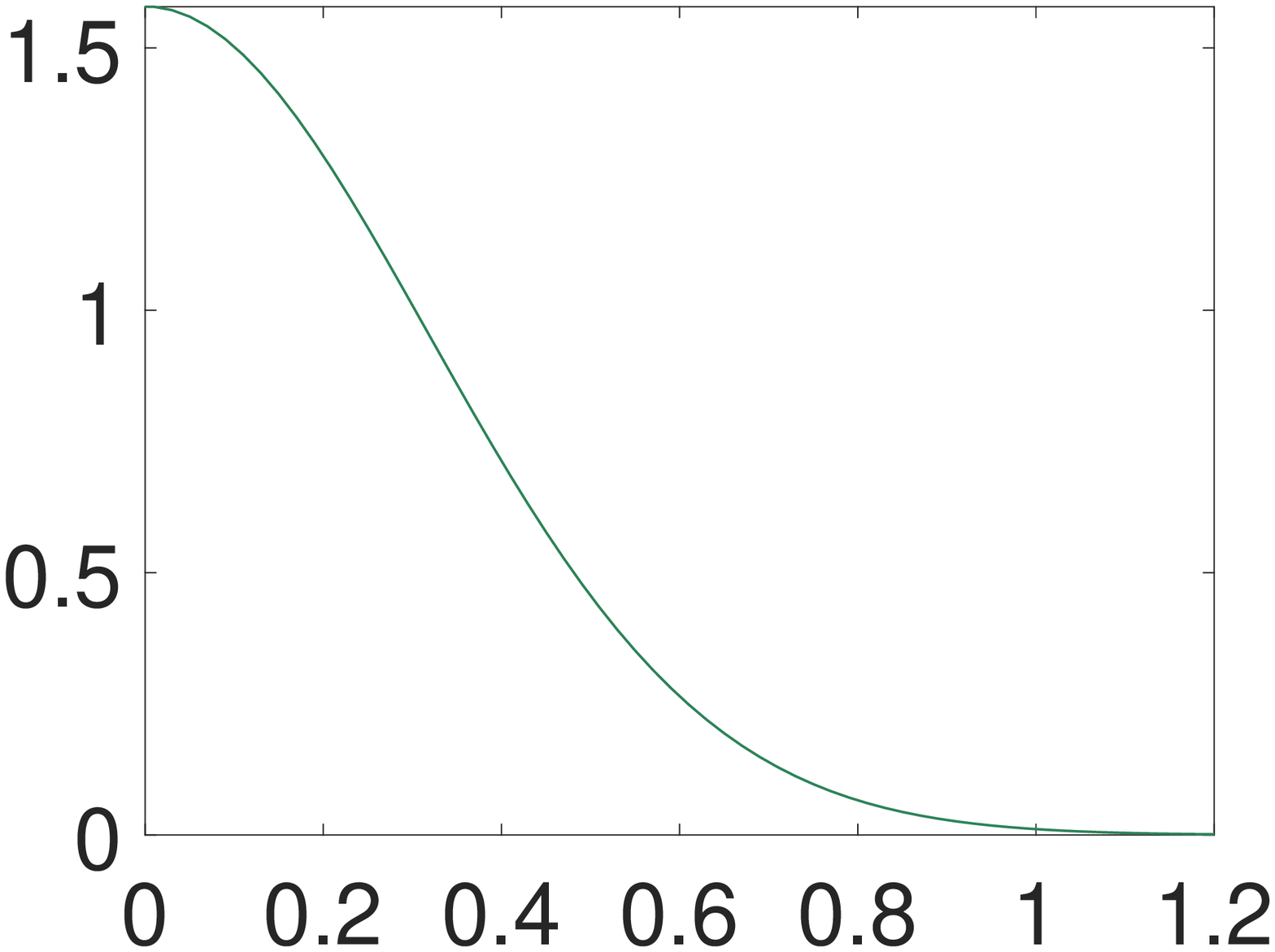}
  \label{fig:diffusion_linesource_LO}}

\caption{Reference solutions for the line source problem.  Heat maps and line outs show the particle concentration $\rho$ in the kinetic ($\epsilon=1$) and diffusive ($\epsilon=10^{-3}$) regimes.}
\label{fig:linesource_ref}
\end{figure}

Similar heat maps and line-outs of for the numerical solution in the kinetic regime ($\epsilon=1$) is shown in Figure~\ref{fig:linesource}. Each of the one-dimensional line-outs are plotted along the $x$-axis and along the direction of $45$ degrees, which shows the most inaccurate part of the solution.
For comparison, the reference kinetic solution is included in all line-out figures.  Plots for the diffusive tests are omitted because the numerical and reference solutions are visually identically.

The run time and relative $L^2$ spatial errors of the particle concentration in both the kinetic and diffusive regime are reported in Table~\ref{table:linesource_time_error}.
The relative $L^2$ spatial error is defined as
\begin{equation}
E := \frac{\|\rho_{\textup{c}}-\rho_{\textup{ref}}\|_{L^2_h(\bbR^2)}}{\|\rho_{\text{ref}}\|_{L^2_{h}(\bbR^2)}}\:,\quad\text{with}\quad\|\rho\|_{L^2_{h}(\bbR^2)}:=\left(\sum_{i,j}\rho_{i,j}^2 h^2\right)^{{1}/{2}}\:,
\label{eq:rel_L2_spat_err}
\end{equation}
where $\rho_{\text{c}}$ is the computed solution, $\rho_{\text{ref}}$ is the reference solution, the summation in \eqref{eq:rel_L2_spat_err} is taken over all $(i,j)$ such that $(x_i,y_j)$ belongs to the uniform spatial mesh, and $h=\dx=\dy$.
Here  are the particle concentrations at $t = t_{\rm{final}}$ of the computed and  solutions with various limiters, respectively.

\begin{figure}[h]
   \captionsetup[subfigure]{justification=centering}
\centering
\subfloat[{\lsrlx}]{
  \includegraphics[width=0.22\linewidth]{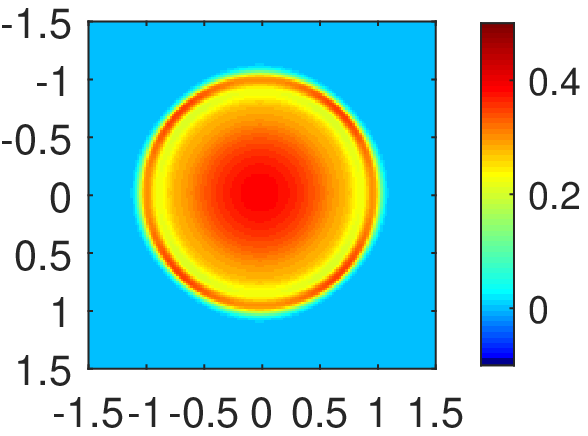}
  \label{fig:LS11_A0_IC0_transport}}~
\subfloat[{\optrlx}]{
\includegraphics[width=0.22\linewidth]{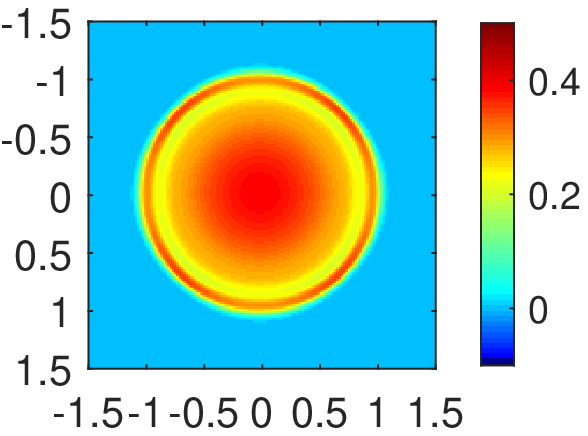}
  \label{fig:OPT11_A0_IC0_transport}}~
\subfloat[{\lspw}]{
  \includegraphics[width=0.22\linewidth]{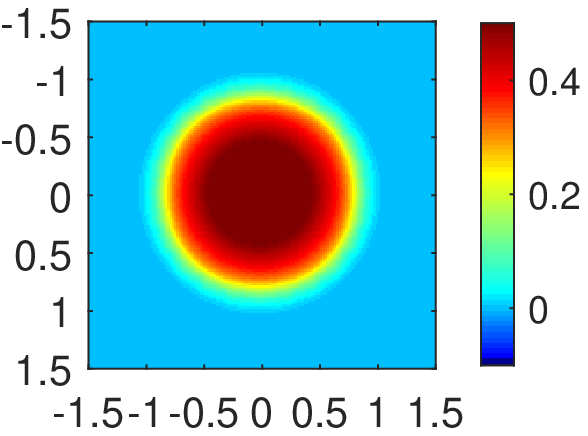}
  \label{fig:LS11_PW_A0_IC0_transport}}~
\subfloat[{\optpw}]{
\includegraphics[width=0.22\linewidth]{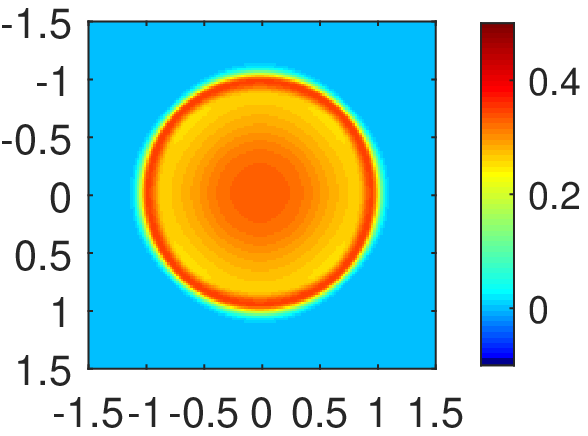}
  \label{fig:OPT11_PW_A0_IC0_transport}}\\

\subfloat[{\lsrlx}]{
  \includegraphics[width=0.22\linewidth]{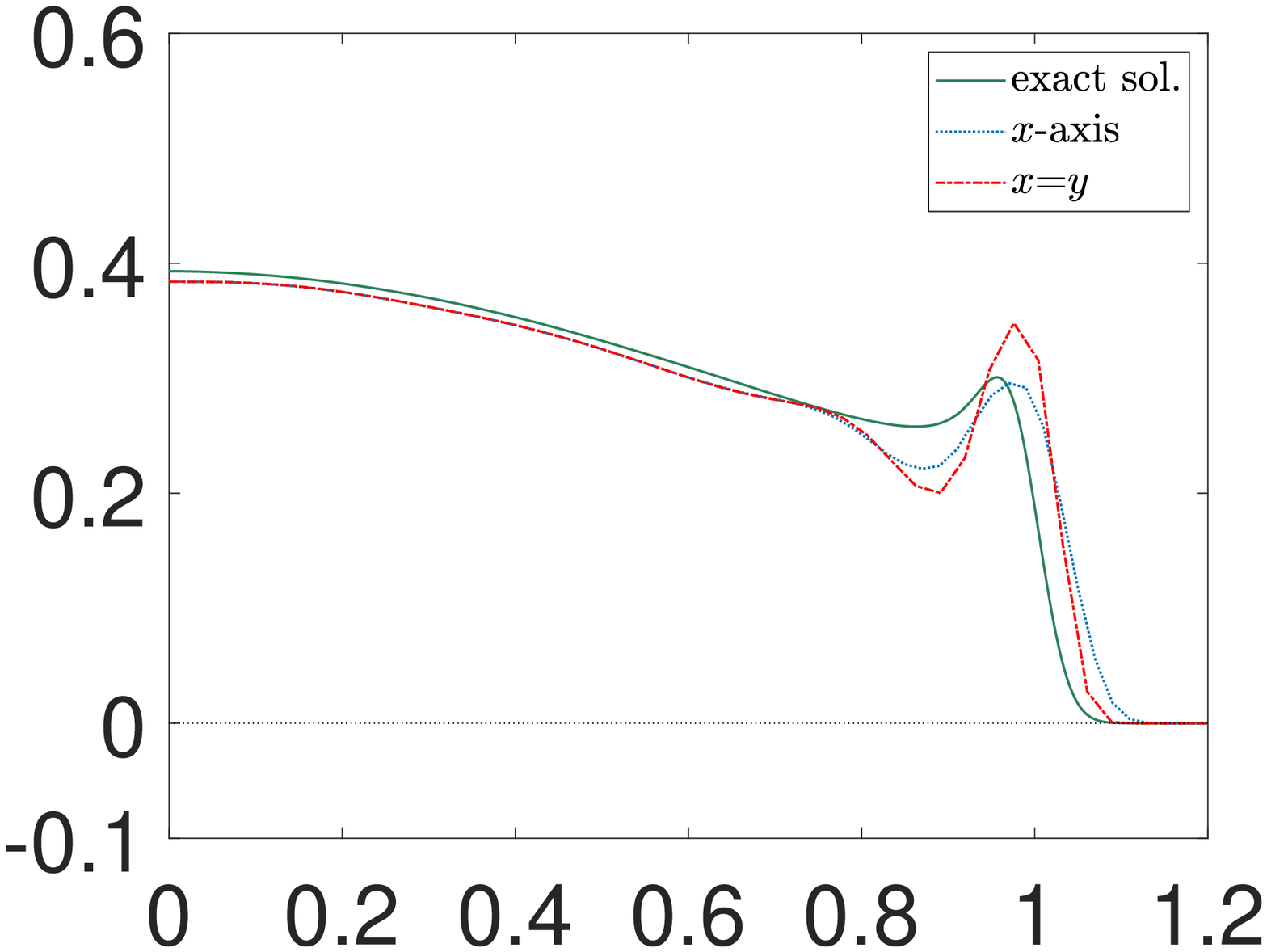}
  \label{fig:LS11_A0_IC0_transport_LO}}~
\subfloat[{\optrlx}]{
  \includegraphics[width=0.22\linewidth]{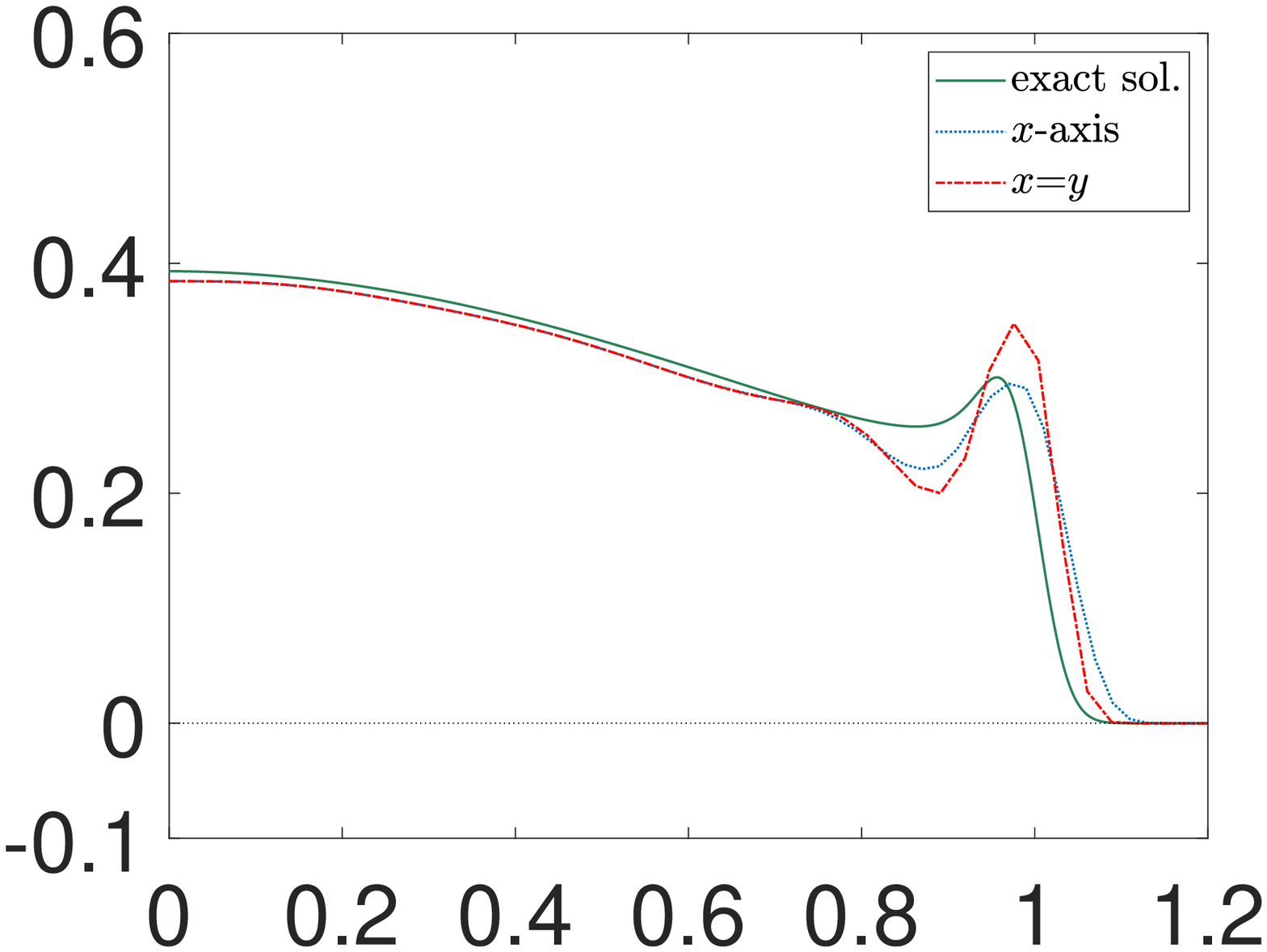}
  \label{fig:OPT11_A0_IC0_transport_LO}}~
\subfloat[{\lspw}]{
  \includegraphics[width=0.22\linewidth]{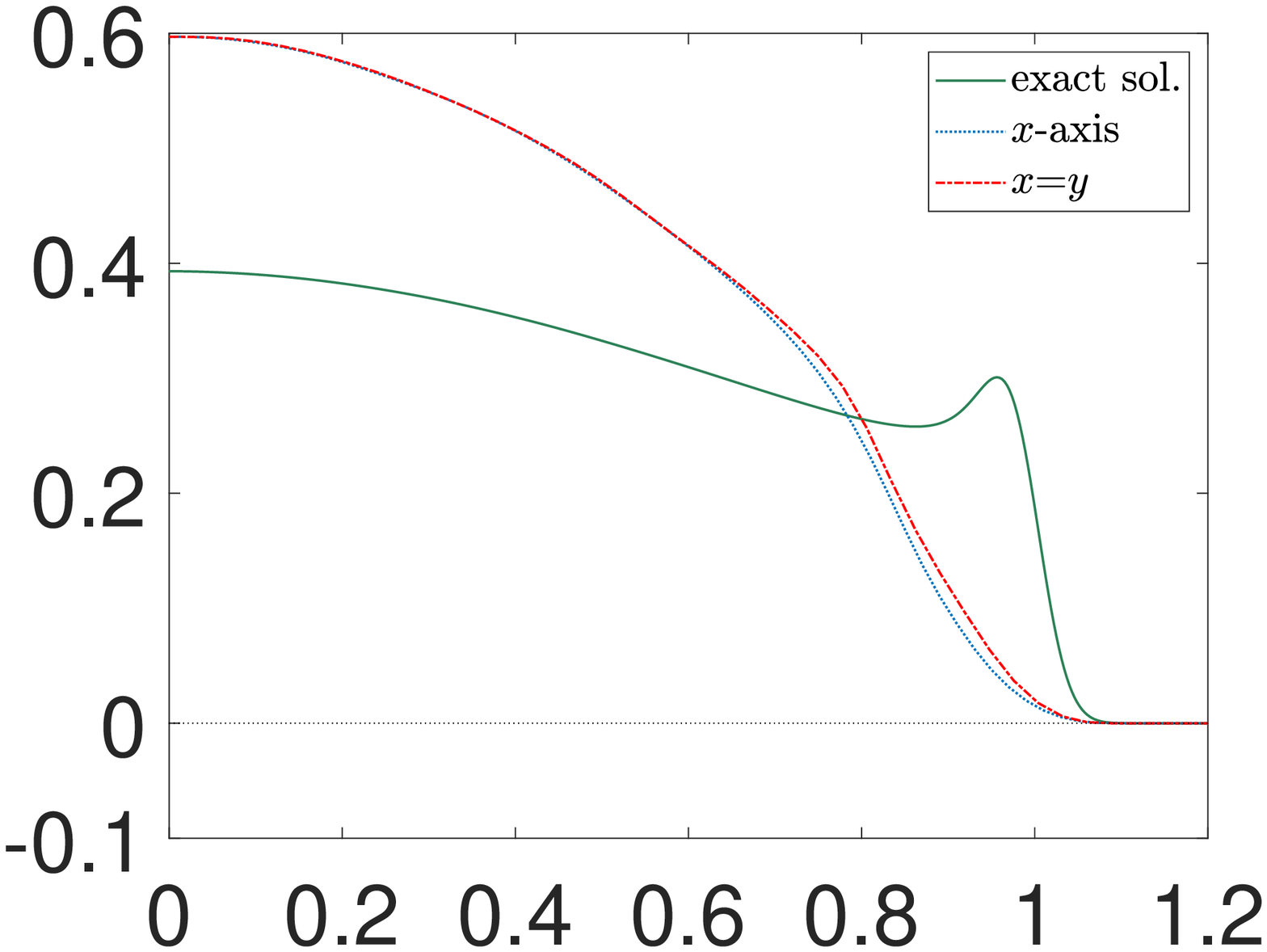}
  \label{fig:LS11_PW_A0_IC0_transport_LO}}~
\subfloat[{\optpw}]{
  \includegraphics[width=0.22\linewidth]{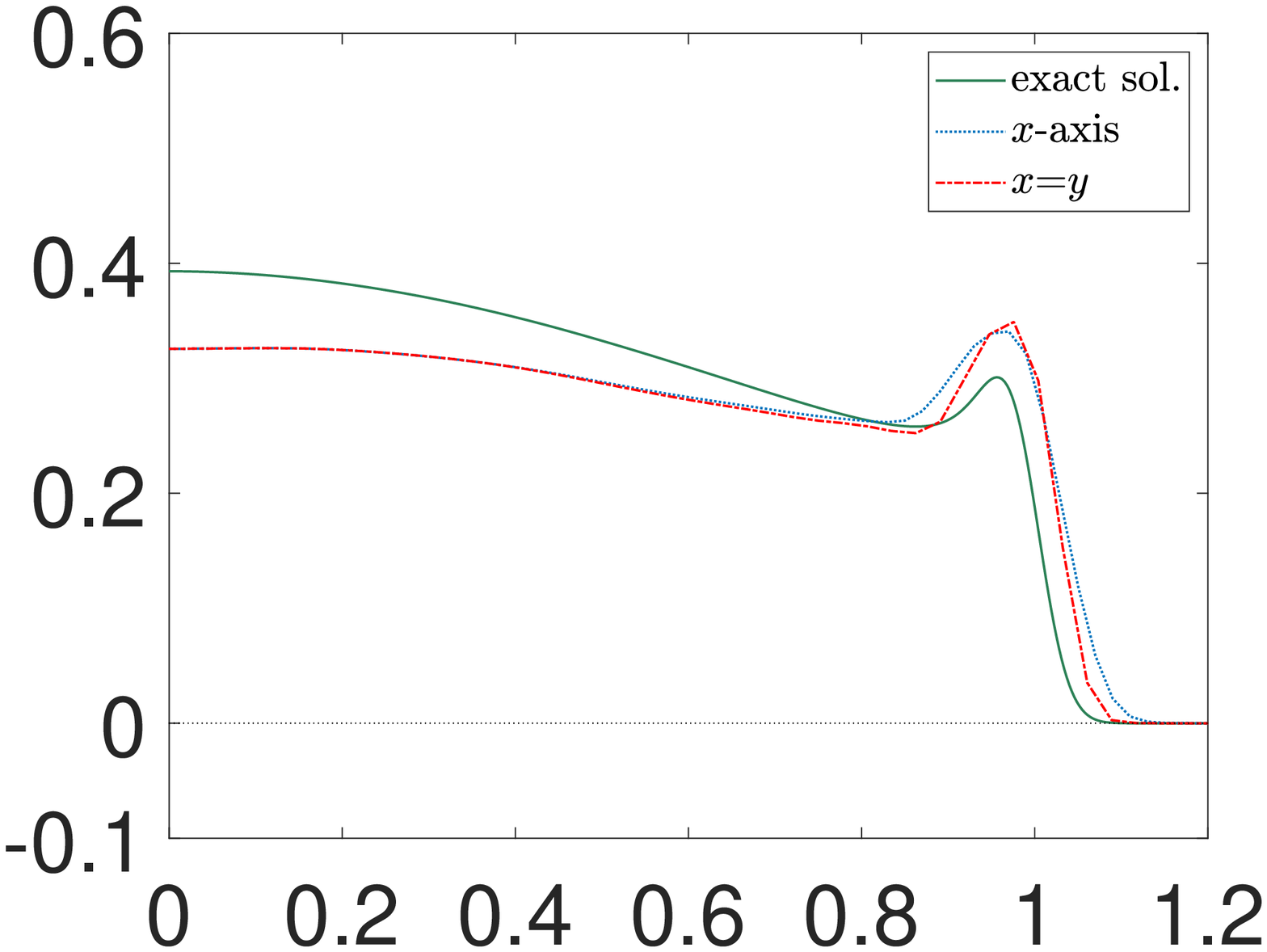}
  \label{fig:OPT11_PW_A0_IC0_transport_LO}}

\caption{Numerical solutions for the line source problem.  Heat maps and line-outs show the particle concentration $\rho$ generated with the {\lsrlx}, {\optrlx}, {\lspw}, and {\optpw} limiters in the kinetic ($\epsilon=1$) regime.
The approximation order of the {\fpn} equations is $N=11$. 
}
\label{fig:linesource}
\end{figure}

\begin{table}[h]
%\footnotesize
\centering
\begin{tabular}{|| c | c || c | c | c | c | c ||}
\toprule
\multicolumn{2}{||c||}{Limiter} &
\texttt{none} & {\lsrlx} & {\optrlx} & {\lspw} & {\optpw} \\
\hline
 {Kinetic} &
 run time   & 280   & 342    & 348   & 413 & 20847 \\
 \cline{2-7}
 $(N=11)$ & $E$ 	& 0.107 & 0.106 & 0.106 & 0.494 & 0.147 \\
\hline
  {Diffusive} &
run time  & 149   &  1040  & 1055  & 1082  &  1044 \\
 \cline{2-7}
$(N=3)$ & $E$ & 0.005 & 0.005 & 0.005 & 0.005  & 0.005 \\
\bottomrule
\end{tabular}
\caption{Run time (sec) and relative $L^2$ spatial error $E$ for the line source problem without a limiter and with the four limiters. In the diffusive regime, the positivity conditions are never violated due to the smooth solution. Thus, all limiters give identical solutions.}
\label{table:linesource_time_error}
\end{table}

In the kinetic regime ($\epsilon=1$), we observe in Figure~\ref{fig:linesource} that with the {\lsrlx}, {\optrlx}, and {\optpw} limiters, the computed solutions are reasonably accurate but slightly more diffusive compared to the reference solution, which we suspect is to the artificial dissipation terms in \eqref{eq:advection_discretization}.
Meanwhile, the solution with the {\lspw} limiter is inaccurate.
Figure~\ref{fig:linesource} also shows that the solutions with the {\lsrlx} and {\optrlx} limiters have some noticeable artifacts that affect the rotational invariance of the solution.
We believe these artifacts come from the axis-dependency of conditions~\eqref{cond:Tx}--\eqref{cond:Dxy2}, as they seem to align with the spatial axes. 
The artifacts are less noticeable in the solution from the {\optpw} limiter, which enforces a stronger pointwise positivity condition and thus applies more damping on the micro coefficients than the {\lsrlx} and {\optrlx} limiters do.

The results in the kinetic regime ($\epsilon=1$) reported in Table~\ref{table:linesource_time_error} indicate that the {\lsrlx} and {\optrlx} limiters give solutions that are as accurate as the unlimited solution.
The {\optpw} limiter gives slightly less accurate solution, while it is about $50$x more computationally expensive than the other three limiters.
In the diffusive regime ($\epsilon=10^{-3}$), the reference solution is sufficiently smooth such that conditions \eqref{cond:Tx}--\eqref{cond:Dxy2} and the pointwise positivity condition are never violated. 
Hence all computed solutions are identical and close to the reference diffusion solution.
We also notice that difference between the computational time of the limited and unlimited cases is more obvious in the diffusive regime. 
This is due to the lower approximation order $N=3$ used in the diffusive tests, which reduces the computational cost in each time step and makes the additional cost of the limiters significant.

\subsection{Problem with non-uniform scattering/absorption}
\label{subsec:nonuniform}

In this section, we test the proposed AP scheme on problems with non-uniform scattering and absorption cross-sections. 
The problem is a modification of the lattice problem formulated in \cite{Brunner-2002} and \cite{Brunner-Holloway-2005}, which is motivated by the geometry of an assembly in a nuclear reactor core. 
As in the original benchmark, a purely scattering medium with strongly absorbing mediums embedded as a checkerboard on a square spatial domain $[-1.5,1.5]\times[-1.5,1.5]$, as shown in Figure~\ref{fig:Lattice_layout}. 
Here the strong absorption regions are colored in white with $\sig{a}=9.9$ and $\sig{s} = 0.1$; the purely scattering regions are colored in black with $\sig{a}=0$ and $\sig{s}=1$.  Unlike the original lattice benchmark, there is no source.  Rather the initial condition, boundary condition, and all other specifics are identical to the ones used in the line source experiments in Section~\ref{subsec:linesource}. The computation is also run on a $150\times 150$ uniform square mesh with the maximum time step allowed by \eqref{eq:Pos_CFL} to final time $t_{\rm{final}} = 1.0$ and $t_{\rm{final}} = 0.1 $ in the kinetic ($\epsilon=1$) and diffusive ($\epsilon=10^{-3}$) regimes, respectively.
For comparison, we compute a reference kinetic solution using the second-order kinetic scheme proposed in \cite{GH12} with a high approximation order $N=37$ on a finely discretized mesh. 
A reference diffusion solution is computed by solving the diffusion equation \eqref{eq:diffusion_2D} with the 9-point finite difference scheme \eqref{eq:macro_limit_2D}.
The reference kinetic and diffusion solutions are shown in Figures~\ref{fig:transport_nonuniform} and \ref{fig:diffusion_nonuniform}, respectively.

\begin{figure}[h]
   \captionsetup[subfigure]{justification=centering}
\centering
\subfloat[Problem Layout]
{  \includegraphics[width=0.26\linewidth]{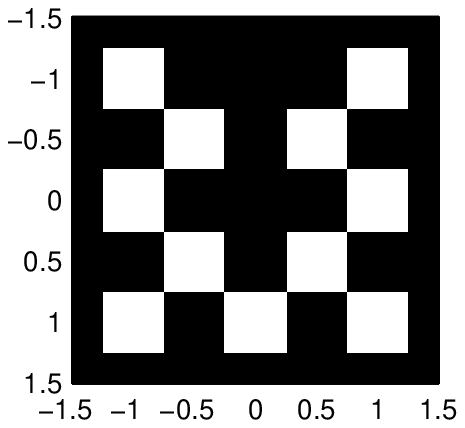}
  \label{fig:Lattice_layout}}~
\subfloat[reference transport solution, $\epsilon=1$, $t=1$][reference solution \\$\epsilon=1$]
{ \includegraphics[width=0.3\linewidth]{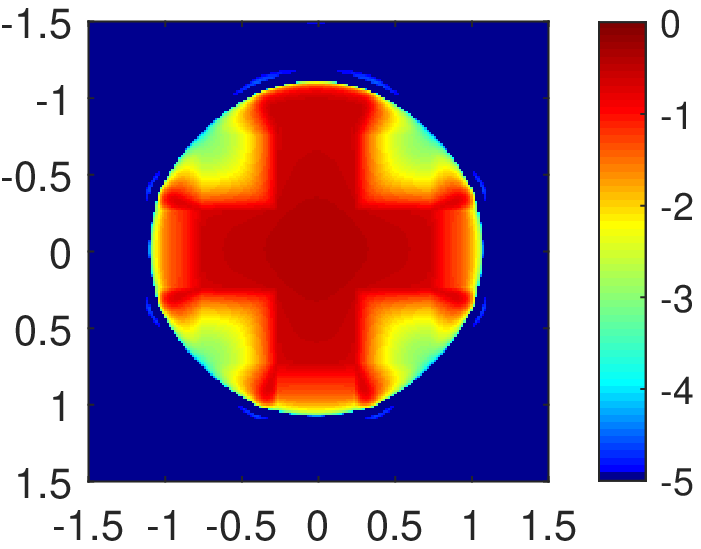}  
\label{fig:transport_nonuniform}}~
\subfloat[reference solution, $\epsilon=10^{-6}$, $t=0.1$][reference solution\\$\epsilon=10^{-3}$]
{  \includegraphics[width=0.3\linewidth]{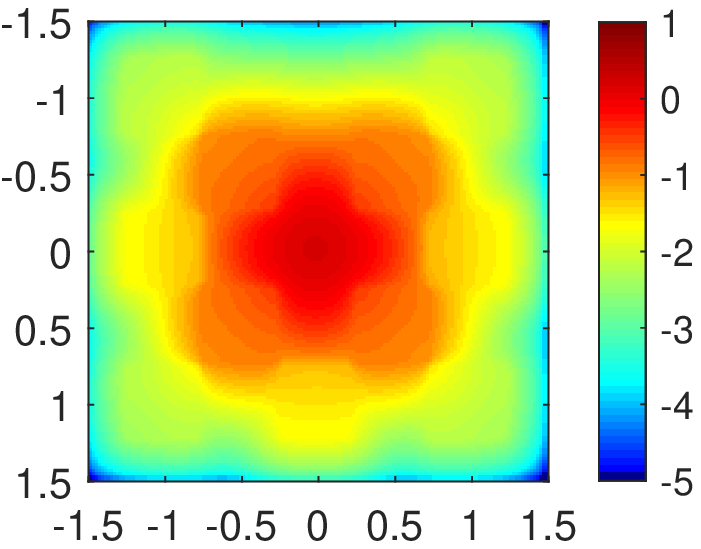}
  \label{fig:diffusion_nonuniform}}\\
  \caption{Problem with non-uniform scattering/absorption cross-sections}
\end{figure}

The run time and relative $L^2$ spatial errors (as defined in \eqref{eq:rel_L2_spat_err}) of the proposed scheme with various limiters are reported in Table~\ref{table:lattice_time_error}.
In the kinetic regime ($\epsilon=1$), the solution from the {\lspw} limiter is still much less accurate than the other solutions. 
The {\optrlx} limiter gives a more accurate solution than the other three limiters, including the expensive {\optpw} limiter.
In the diffusive regime ($\epsilon=10^{-3}$), all limiters gives identical solutions since conditions \eqref{cond:Tx}--\eqref{cond:Dxy2} and the pointwise positivity condition are always satisfied.
Similar to the line source results, the difference in the computational time of the limited and unlimited solutions is more significant in the diffusive regime due to the lower approximation order $N=3$.
Here the computed solutions have relatively large errors in the diffusion regime when compared to the computed diffusion solutions in the line source case.
It follows from \eqref{eq:gamma_max_def} that smaller $\sig{s}^{\min}$ leads to stronger artificial dissipation.
Thus, we suspect that the stronger artificial dissipation introduced in this non-uniform problem ($\sig{s}^{\min}=0.1$) leads to less accurate diffusion solutions than the ones in the line source case ($\sig{s}^{\min}=1$).

To confirm that the computed solutions actually converge to the reference diffusion solution as $\epsilon\to0$, we sample $\epsilon$ from $10^{-3}$ to $10^{-7}$, and report the $L^2$ spatial error $E$ and its convergence order $\nu$ at each sample of $\epsilon$.
Let $\epsilon_i$ denote the samples of $\epsilon$, the order $\nu$ is computed by $\nu := \log\left(\frac{E_{\epsilon_i}}{E_{\epsilon_{i+1}}}\right)\log \left(\frac{\epsilon_i}{\epsilon_{i+1}}\right)^{-1}$, with $E_{\epsilon_i}$ the $L^2$ spatial error when $\epsilon=\epsilon_i$.
The convergence result is reported in Table~\ref{table:lattice_eps_cvgce}, which shows first-order convergence of spatial error.
Since all limiters are effectively inactive when $\epsilon$ is small, we only report one set of the $L^2$ spatial errors in Table~\ref{table:lattice_eps_cvgce}.

\begin{table}[h]
\centering
\begin{tabular}{|| c | c || c | c | c | c | c ||}
\toprule
\multicolumn{2}{||c||}{Limiter} &
\texttt{none} & {\lsrlx} & {\optrlx} & {\lspw} & {\optpw} \\
\hline
{Kinetic} &
 run time   & 363 & 391  & 406 & 483 & 24263 \\
 \cline{2-7}
$(N=11)$ & $E$ 	& 0.092 & 0.090 & 0.083 & 0.501 & 0.132 \\
\hline
  {Diffusive} &
 run time  & 1481  & 11718  &  11922 &  11754 & 11662 \\
 \cline{2-7}
$(N=3)$ & $E$ & 0.218 & 0.218 & 0.218 & 0.218  & 0.218 \\
\bottomrule
\end{tabular}
\caption{Run time (sec) and relative $L^2$ spatial errors at  for the computed solutions without limiter and with the four limiters on the problem with non-uniform scattering/absorption cross-sections in the kinetic regime ($\epsilon=1$, $t_{\rm{final}}=1$) and diffusive regime ($\epsilon=10^{-3}$, $t_{\rm{final}}=1$). In the diffusive regime, the positivity conditions are never violated due to the smoothness of the solution. Thus, all limiters give solutions that are identical to the one without limiter.}
\label{table:lattice_time_error}
\end{table}

\begin{table}[h]
\centering
\begin{tabular}{|| c || c | c | c | c |c ||}
\toprule
{$\epsilon$} &
{$10^{-3}$} & {$10^{-4}$} & {$10^{-5}$} & {$10^{-6}$} & {$10^{-7}$} \\
\hline
{$E$} & 
2.2e-1 & 1.9e-1  & 3.3e-2 & 3.9e-3 & 4.1e-4 \\
\hline
{$\nu$} & 
--- & 0.06 & 0.72 & 0.93 & 0.98 \\
\bottomrule
\end{tabular}
\caption{Convergence of the $L^2$ spatial error as $\epsilon\to 0$ in the diffusion regime at $t_{\rm{final}}=1$. The errors $E$ and the convergence orders $\nu$ are reported for $\epsilon$ sampled between $10^{-3}$ and $10^{-7}$.} 
\label{table:lattice_eps_cvgce}
\end{table}

\subsection{Space-time accuracy}
\label{subsec:space_time}

As a final test, we investigate the order of accuracy of the proposed AP scheme in the kinetic ($\epsilon=1$), transition ($\epsilon=10^{-2}$), and diffusive ($\epsilon=10^{-4}$)%
\footnote{Here we choose $\epsilon=10^{-4}$ for the diffusive regime to ensure that the problem stays in the diffusive regime even when the space-time mesh is refined.} 
regimes.
As in previous tests, we truncate the spatial domain to a $[-1.5,1.5]\times[-1.5,1.5]$ square centered at the origin and impose an artificial zero boundary condition.
The computation is run on uniform square meshes of size $20\times20$ to $160\times160$. 
The reference solutions are computed on a $640\times640$ uniform square mesh.
The final time is $t_{\rm{final}} = 1$, $t_{\rm{final}} = 0.05$, and $t_{\rm{final}} = 0.01$ in the kinetic, transition, and diffusive regimes, respectively.
Since the steep Gaussian initial condition \eqref{eq:Gaussian_ic} may limit the observed convergence order, we test the scheme with a ``gradual'' Gaussian initial condition, which takes the same form as \eqref{eq:Gaussian_ic} with variance $\varsigma^2=5\times10^{-3}$.

All parameter values used in the space-time convergence tests are identical to those listed in Section~\ref{subsec:linesource}, except that we choose the approximation order $N=5$ (instead of 11) for the kinetic and transition tests and $N=3$ for the diffusive tests.
Similar to \eqref{eq:rel_L2_spat_err}, we define the relative $L^2$ space-time error $E_{h}$ by
\begin{equation}
E_{h}:=\|\rho_{h}-\rho_{\textup{ref}}\|_{L^2_{h_{\textup{ref}}}(\bbR^2)}/\|\rho_{\textup{ref}}\|_{L^2_{h_{\textup{ref}}}(\bbR^2)}\:, 
\end{equation}
where $\rho_{h}$ is the particle concentration of the solution computed by the proposed scheme with spatial grid size $h=\dx=\dy$, $\rho_{\textup{ref}}$ is the reference particle concentration, and $h_{\textup{ref}}$ is the grid size of the reference solution.
Since an analytic solution to this problem is not available, we use the computed solutions on a $640\times 640$ spatial mesh as the reference solutions.

Table~\ref{table:space-time_cvgce_L2} reports the $L^2$ space-time errors $E_h$ and observed convergence orders $\nu$ for the proposed scheme with {\lsrlx} and {\optrlx} limiters in the three regimes.
The order $\nu$ is computed by
\begin{equation}
%\nu := \log\left(\frac{E^p_{h_i}}{E^p_{h_{i+1}}}\right)\log \left(\frac{h_i}{h_{i+1}}\right)^{-1}\:,\quad i=1,\dots,4,
\nu := \log\left(\frac{E_{h_i}}{E_{h_{i+1}}}\right)\log \left(\frac{h_i}{h_{i+1}}\right)^{-1}\:,\quad i=1,\dots,5,
\end{equation}
where $h_i:=\dx=\dy$ is defined by the mesh sizes given in Table~\ref{table:space-time_cvgce_L2}.

The results in Table~\ref{table:space-time_cvgce_L2} show that the observed order of space-time accuracy is between one and two in the kinetic and transition regimes. 
In the diffusive regime, the proposed scheme shows second-order accuracy due to the refined time step.
Also, there is no noticeable difference between the results from the two limiters, since with the gradual Gaussian initial condition the limiters are rarely active. 
Based on the theoretical results for the one-dimensional version of the proposed scheme given in \cite{Laiu-Hauck-2018}, we expect the scheme to be at least first-order accurate in all three regimes when the time step satisfies \eqref{eq:Pos_CFL} and assumption (ii) in Section~\ref{subsec:spatial_discretization_AP}.
The time steps in these tests satisfy both conditions, and the numerical results in Table~\ref{table:space-time_cvgce_L2} confirm the theoretical estimate.

\begin{table}[h]
\scriptsize
\begin{center}
\begin{tabular}{ccccccccccccc}
 & \multicolumn{4}{c}{Kinetic, $\epsilon=1$} 
 & \multicolumn{4}{c}{{Transition, $\epsilon=10^{-2}$}} 
 & \multicolumn{4}{c}{{Diffusive, $\epsilon=10^{-4}$}}\\
   \cmidrule(r){2-5} \cmidrule(r){6-9} \cmidrule(r){10-13} 
 & \multicolumn{2}{c}{{\lsrlx}} & 
   \multicolumn{2}{c}{{\optrlx}}
 & \multicolumn{2}{c}{{\lsrlx}} & 
   \multicolumn{2}{c}{{\optrlx}}
 & \multicolumn{2}{c}{{\lsrlx}} & 
   \multicolumn{2}{c}{{\optrlx}}\\   
   \cmidrule(r){2-3} \cmidrule(r){4-5} \cmidrule(r){6-7}
\cmidrule(r){8-9} \cmidrule(r){10-11} \cmidrule(r){12-13} 
Mesh 
& $E_{h}$      & $\nu$ 
& $E_{h}$      & $\nu$ 
& $E_{h}$      & $\nu$ 
& $E_{h}$      & $\nu$ 
& $E_{h}$      & $\nu$ 
& $E_{h}$      & $\nu$\\ \midrule

 $20^2$ & 4.2e-1 & --- & 4.2e-1 &  ---   & 5.8e-1 & --- & 5.8e-1 &  --- 	& 9.2e-2 & ---  	& 9.2e-2 &  ---  \\ 
 $40^2$ & 2.7e-1 & 0.6 & 2.7e-1 &  0.6   & 4.5e-1 & 0.4 & 4.5e-1 & 0.4 	& 4.2e-2 & 1.1 	& 4.2e-2 &  1.1 \\     	
 $80^2$ & 1.1e-1 & 1.2 & 1.1e-1 &  1.2   & 2.5e-1 & 0.8 & 2.5e-1 & 0.8   	& 7.0e-3 & 2.6 	& 7.0e-3 &  2.6 \\ 
$160^2$ & 2.8e-2 & 2.0 & 2.8e-2 &  2.0   & 7.0e-2 & 1.8 & 7.0e-2 & 1.8   	& 1.8e-3 & 2.0	& 1.8e-3 &  2.0 \\ 
\end{tabular}
\end{center}
\caption{Convergence of space-time errors -- The space-time errors $E_h$ and observed convergence orders $\nu$ are reported. The spatial mesh sizes are listed in the first column. 
We observe at least first order in all three regimes, which confirms the theoretical estimate.
There is no noticeable difference in the results from the two limiters. 
}
\label{table:space-time_cvgce_L2}
\end{table}

\section{Conclusions and discussion}
\label{sec:conclusion}
We have proposed a new positive asymptotic preserving scheme for solving the {\fpn} equations, an approximation to the linear kinetic transport equations, in two space dimensions.
The scheme applies a micro-macro decomposition to the {\fpn} equations and solves the resulting system with a suitable semi-implicit temporal discretization and a special finite difference spatial discretization. 
We give sufficient conditions under which the proposed scheme preserves positivity of particle concentrations and we show that these sufficient conditions are satisfied under a reasonable time-step restriction with the imposition of {\limname} limiters.
We test the proposed scheme on the notoriously difficult line source benchmark problem as well as a multiscale lattice problem. 
Numerical results confirm that in both the kinetic (large mean-free-path) and diffusive (small mean-free-path) regimes, the scheme gives accurate solutions and preserves the nonnegativity of particle concentrations.
The space-time convergence result shows that the accuracy of the proposed scheme is between first and second order in the kinetic and transition regimes, and is second-order in the diffusive regime.

The uniform stability and accuracy analysis of the proposed scheme is presented in \cite{Laiu-Hauck-2018} for the one-dimensional case. 
The analysis indicates that the accuracy of this scheme is limited by the first-order semi-implicit temporal discretization.
To achieve higher order of accuracy, it is possible to implement the proposed finite difference method and the {\limname} limiters together with a second-order implicit-explicit (IMEX) temporal discretization, which has been considered for stiff ordinary differential equations \cite{Chertock-Cui-Kurganov-Wu-2015} and the stiff BGK equation \cite{Hu-Shu-Zhang-2018}.
However, it is known that IMEX schemes requires a restrictive time step either to preserve positivity of the solution \cite{Hu-Shu-Zhang-2018, Higueras-Roldan-2006} or to maintain the strong-stability-preserving (SSP) property for the implicit update \cite{Conde-Gottlieb-Grant-Shadid-2017}.
It is also not clear if IMEX schemes resolve the diffusion limit correctly, since only the Euler limit is considered in \cite{Hu-Shu-Zhang-2018}.
On the other hand, the discontinuous Galerkin (DG) spatial discretization has been used together with the micro-macro decomposition to develop high order AP schemes for the linear kinetic transport equations \cite{Jang-Li-Qiu-Xiong-2015, Jang-Li-Qiu-Xiong-2014} and the BGK equation \cite{Xiong-Jang-Li-Qiu-2015}.
To develop a higher order positive-preserving AP scheme, it is also possible to apply some modified {\limname} limiters on these schemes to enforce positivity.
However, the derivation of physical positivity conditions under the DG spatial discretization is not straightforward.

Other potential future work includes: a rigorous stability and accuracy analysis in the multi-dimensional case for the proposed scheme; an extension of the proposed scheme to three dimensions, where we believe the naive extension suffers from the issue of alternating grids, and a modified extension, such as the one described in Section~\ref{subsec:3D}, is needed; the application of the proposed scheme on other equations, such as the the Vlasov-Poisson equation and the linear Boltzmann equation, where the multiscale behavior and the positivity of the solution are of interest.

{
\appendix
\section{Calculation of diffusion matrices}
\label{appendix:diffusion_calculation}
In this appendix, we provide detailed calculations in the derivation of diffusion matrices $\Qm$ and $\QM$ in \eqref{eq:Qm_def} and \eqref{eq:QM_def}.
We first write the submatrices of $\Qm$ in \eqref{eq:Qm_def} as $\vint{\mM\mm^T\Omega_\alpha} \GammaM \vint{\mm\mm^T\Omega_\beta}$, for $\alpha=x,y$ and $\beta = x,y$. 
For $\alpha=x,y$, let $k_\alpha$ denote the index such that $c_\alpha\mm_{k_\alpha}= \Omega_\alpha$ with some nonzero constant $c_\alpha\in\bbR$.
We then have
\begin{equation}
\begin{alignedat}{2}
\vint{\mM\mm^T\Omega_\alpha} \GammaM \vint{\mm\mm^T\Omega_\beta} &= \mM \vint{\mm^T c_\alpha\mm_{k_\alpha}} \GammaM \vint{\mm\mm^T\Omega_\beta}
= \mM c_\alpha \GammaM_{k_\alpha,k_\alpha} \vint{\mm_{k_\alpha}\mm^T \Omega_\beta}\\
&= \GammaM_{k_\alpha,k_\alpha} \vint{\mM \mm^T c_\alpha \mm_{k_\alpha} \Omega_\beta} = \GammaM_{k_\alpha,k_\alpha} \vint{\mM \mm^T \Omega_\alpha \Omega_\beta}\:,
\end{alignedat}
\end{equation}
where the second equality follows from the fact that since entries of $\mm$ are orthonormal, $\vint{\mm^T \mm_{k_\alpha}}$ is a vector of all zeros except its $k_\alpha$-th entry, which takes value one.
Further, we observe that $\G_{k_x,k_x}=\G_{k_y,k_y}$, which follows directly from the definition of $\G$ in \eqref{eq:Gamma_def} and the definition of the filtering matrix $F$ introduced in \eqref{eq:fpn}.
Thus, we denote $\g:=\G_{k_x,k_x}=\G_{k_y,k_y}$ in \eqref{eq:Qm_def} and \eqref{eq:QM_def}.
The equalities in \eqref{eq:Qm_def} is now verified, and the equalities in \eqref{eq:QM_def} can be shown similarly.

\section{Lower bound of the artificial dissipation operator}
\label{appendix:minmod_bound}

In this appendix, we prove a lower bound needed in \eqref{eq:ddx_bound} in the proof of Theorem~\ref{thm:positivity}.
Specifically we show that for any nonnegative function $w$ defined on the spatial mesh,
\begin{equation}
\ddx(\wMc) \geq \frac{1}{2\dx^4} \left(\frac{1}{\Theta} \wMe - 4 \wMc + \frac{1}{\Theta}\wMw\right)\:,
\label{eq:minmod_lower_bound}
\end{equation}
where $\ddx$ is the artificial dissipation operator defined in \eqref{eq:ddx_def} and $\Theta=\frac{1}{2-\theta}$.
From \eqref{eq:ddx_def} and \eqref{eq:ddx_edge_values},
\begin{equation}
%-\dx^4\ddx(\wMc)&= \left((\wMe - \frac{\dx}{2} s_{i+1,j}^x) - (\wMc + \frac{\dx}{2} s_{i,j}^x)\right)\\
%&\qquad -\left((\wMc - \frac{\dx}{2} s_{i,j}^x) - (\wMw + \frac{\dx}{2} s_{i-1,j}^x)\right)\\
\dx^4\ddx(\wMc)= \wMe - \frac{\dx}{2} s_{i+1,j}^x - 2\wMc + \wMw + \frac{\dx}{2} s_{i-1,j}^x\:,
\label{eq:minmod_pf_1}
\end{equation}
where the slope $s_{i,j}^x$ is defined in \eqref{eq:slope_minmod}.
To verify \eqref{eq:minmod_lower_bound}, we first consider the case that $\wMe \geq \wMc \geq \wMw$, which, together with \eqref{eq:slope_minmod}, leads to
\begin{equation}
s_{i+1,j}^x\leq \theta\frac{\wMe - \wMc}{\dx} \quand s_{i-1,j}^x\geq 0\:.
\end{equation}
Thus, \eqref{eq:minmod_pf_1} gives that when $\wMe \geq \wMc \geq \wMw$,
\begin{equation}
\dx^4\ddx(\wMc) \geq \left(1-\frac{\theta}{2}\right)\wMe - \left(2-\frac{\theta}{2}\right)\wMc + \wMw\:. 
\end{equation}
By applying similar arguments on other cases, we show that $\dx^4\ddx(\wMc)$ is bounded from below by
\begin{equation}
%\begin{alignedat}{2}
%\dx^2\ddx(w) &\geq \half\hwMe - \frac{3}{2}\wMc + \hwMw \quad \text{when } \hwMe \geq \wMc \geq \hwMw\\
%\dx^2\ddx(w) &\geq \half\hwMe - \wMc + \half\hwMw \quad \text{when } \hwMe \geq \wMc\:, \wMc\leq \hwMw\\
%\dx^2\ddx(w) &\geq \hwMe - 2\wMc + \hwMw \quad \text{when } \hwMe \leq \wMc\:, \wMc \geq \hwMw\\
%\dx^2\ddx(w) &\geq \hwMe - \frac{3}{2}\wMc + \half\hwMw \quad \text{when } \hwMe \leq \wMc \leq \hwMw
%\end{alignedat}
%\dx^4\ddx(\wMc) \geq\begin{cases}
% \half \wMe - \frac{3}{2}\wMc + \wMw\:,  &\text{when } \wMe \geq \wMc \geq \wMw\\
% \half\wMe - \wMc + \half\wMw\:,  &\text{when } \wMe \geq \wMc\:, \wMc< \wMw\\
% \wMe - 2\wMc + \wMw\:,  &\text{when } \wMe < \wMc\:, \wMc \geq \wMw\\
% \wMe - \frac{3}{2}\wMc + \half\wMw\:,  &\text{when } \wMe < \wMc < \wMw
\begin{cases}
 \left(1-\frac{\theta}{2}\right) \wMe - \left(2-\frac{\theta}{2}\right)\wMc + \wMw\:,  &\text{if } \wMe \geq \wMc \geq \wMw\\
 \left(1-\frac{\theta}{2}\right) \wMe - \left(2-\theta\right)\wMc + \left(1-\frac{\theta}{2}\right)\wMw\:,  &\text{if } \wMe \geq \wMc\:, \wMc< \wMw\\
 \wMe - 2\wMc + \wMw\:,  &\text{if } \wMe < \wMc\:, \wMc \geq \wMw\\
 \wMe - \left(2-\frac{\theta}{2}\right)\wMc + \left(1-\frac{\theta}{2}\right)\wMw\:,  &\text{if } \wMe < \wMc < \wMw
\end{cases}\:.
\end{equation}
Since $\wMe$, $\wMc$, and $\wMw$ are all nonnegative and $\theta\in(1,2)$, we conclude that \eqref{eq:minmod_lower_bound} holds.

}

\bibliographystyle{siamplain}
\bibliography{PFPN_ref}

\end{document}